\documentclass[a4paper,12pt,reqno]{amsart}
\usepackage{amssymb}
\usepackage{amsmath}
\usepackage{mathtools}
\usepackage{enumitem}
\usepackage{mathrsfs}
\usepackage[all]{xy}
\usepackage{mathtools}
\usepackage{hyperref}
\usepackage{url}
\usepackage{bbm}
\usepackage{longtable}

\usepackage[OT2,T1]{fontenc}
\DeclareSymbolFont{cyrletters}{OT2}{wncyr}{m}{n}
\usepackage[british]{babel}

\usepackage[margin=1.3in]{geometry}
\numberwithin{equation}{section} \numberwithin{figure}{section}

\DeclareMathOperator{\Pic}{Pic} 
\DeclareMathOperator{\Gal}{Gal} 
\DeclareMathOperator{\Aut}{Aut}

\DeclareMathOperator{\Hom}{Hom} \DeclareMathOperator{\re}{Re}
\DeclareMathOperator{\im}{Im}   
\DeclareMathOperator{\vol}{vol}

\DeclareMathOperator{\Tor}{Tor} \DeclareMathOperator{\Ext}{Ext}
\DeclareMathOperator{\rad}{rad}
\DeclareMathOperator{\Res}{R}

\DeclareMathSymbol{\Sha}{\mathalpha}{cyrletters}{"58}

\newcommand{\OO}{\mathcal{O}}

\newcommand\exterior{\wedge^2}
\newcommand{\legendre}[2]{\left(\frac{#1}{#2}\right)}
\newcommand{\dual}[1]{{#1}^{\wedge}}
\newcommand{\QZdual}[1]{{#1}^{\sim}}
\newcommand{\Cdual}[1]{\widehat{#1}}
\newcommand{\Const}{\mathbbm{1}}

\newcommand\ZZ{\mathbb{Z}}
\newcommand\NN{\mathbb{N}}
\newcommand\QQ{\mathbb{Q}}

\newcommand\CC{\mathbb{C}}
\newcommand\GG{\mathbb{G}}

\newcommand\Gm{\GG_\mathrm{m}}

\newcommand{\Adele}{\mathbf{A}}

\newcommand{\fp}{\mathfrak{p}}
\newcommand{\gextk}{\ensuremath{G\text{-ext}(k)}}
\newcommand{\pair}[2]{\ensuremath{\langle #1, #2 \rangle}}

\newtheorem{lemma}{Lemma}

\newtheorem{theorem}[lemma]{Theorem}
\newtheorem{proposition}[lemma]{Proposition}

\theoremstyle{definition}

\newtheorem{remark}[lemma]{Remark}

\usepackage[usenames,dvipsnames]{color}

\numberwithin{lemma}{section}

\begin{document}

\title[The Hasse norm principle for abelian extensions]
{The Hasse norm principle for abelian extensions}

\author{\sc Christopher Frei}
\address{Christopher Frei\\
School of Mathematics\\
University of Manchester\\
Oxford Road\\
Manchester\\
M13 9PL\\
UK.}
\email{christopher.frei@manchester.ac.uk}
\urladdr{https://personalpages.manchester.ac.uk/staff/christopher.frei}

\author{\sc Daniel Loughran}
\address{Daniel Loughran \\
School of Mathematics\\
University of Manchester\\
Oxford Road\\
Manchester\\
M13 9PL\\
UK.}
\email{daniel.loughran@manchester.ac.uk}
\urladdr{https://sites.google.com/site/danielloughran}

\author{\sc Rachel Newton}
\address{Rachel Newton\\
Department of Mathematics and Statistics\\ 
University of Reading\\
Whiteknights\\
PO Box 220\\
Reading RG6 6AX\\ 
UK.}
   \email{r.d.newton@reading.ac.uk}
\urladdr{https://racheldominica.wordpress.com/}


\begin{abstract}
	We study the distribution of abelian extensions of bounded discriminant of a number field $k$
	which fail the Hasse norm principle. For example, we classify those finite abelian groups $G$ for which a positive proportion 
	of \mbox{$G$-extensions} of $k$ fail the Hasse norm principle.
	We obtain a similar classification for the failure of weak approximation for the associated norm one tori. These results involve counting abelian extensions of bounded discriminant
	with infinitely many local conditions imposed, which we achieve using tools from harmonic analysis, building on
	work of Wright.

\end{abstract}

\maketitle

\thispagestyle{empty}

\tableofcontents

\section{Introduction} 
\subsection{The Hasse norm principle}
Let $K/k$ be an extension of number fields and denote by $N_{K/k}: \Adele^*_K \to \Adele^*_k$ the norm
map from the ideles of $K$ to the ideles of $k$. We say that $K/k$ satisfies the \emph{Hasse norm principle}
if $N_{K/k} K^* = N_{K/k} \Adele_K^* \cap k^*$. This property takes its name from a theorem of Hasse,
which states that cyclic extensions always satisfy the Hasse norm principle.
In this paper, we fix a finite abelian group $G$ and ask how many Galois extensions of $k$ with Galois group $G$ satisfy the Hasse norm principle. This idea of studying Hasse principles in families
has attracted a lot of attention recently in other settings; see for example \cite{Bha14c}, \cite{BB14a}, \cite{BB14b}, \cite{BN15}.

Let us begin by considering the case of biquadratic extensions, 
which is the simplest case where the Hasse norm principle can fail.
A biquadratic extension fails the Hasse norm principle
if and only if all its decomposition groups are cyclic. This condition may be viewed in an entirely 
elementary way over $\QQ$, as we now explain. 
Let $a,b \in \ZZ$  and assume for simplicity that $a,b$ are coprime, square-free, that $a,b \neq 1$
and that $a,b \equiv 1 \bmod 4$.
Then the biquadratic extension $\QQ(\sqrt{a}, \sqrt{b})/\QQ$ fails the Hasse norm principle
if and only if
\begin{equation} \label{eqn:biquadratic}
\legendre{a}{p} = 1 \,\, \text{ for all } p \mid b \quad \text{ and } \quad 
\legendre{b}{p} = 1 \,\, \text{ for all } p \mid a,
\end{equation}
where $\legendre{\cdot}{\cdot}$ denotes the Legendre symbol.
A simple sieving argument shows that, for fixed $a$ as above, $0\%$ of 
all such $b$ satisfy \eqref{eqn:biquadratic}. In particular, it leads one to guess that $0\%$
of all biquadratic extensions, when ordered by discriminant, fail the Hasse norm principle.
Our first result shows that this is indeed the case, and moreover applies to a slightly larger
class of abelian extensions.

In order to state this result, we shall need some definitions. Let $G$ be a finite group and let $k$ be a number field with a 
fixed choice of algebraic closure $\bar{k}$. By a \emph{$G$-extension} of $k$, we mean a surjective continuous homomorphism $\varphi: \Gal(\bar{k}/k)\to G$. Clearly, a $G$-extension is the same as a pair $(K/k, \psi)$, where $K/k$ is a Galois extension contained in $\bar{k}$ and $\psi : \Gal(K/k)\to G$ is an isomorphism. We write $K_\varphi/k$ for the extension corresponding to $\varphi$ and $\Delta(\varphi)$ for the absolute norm of the relative discriminant ideal of $K_\varphi$ over $k$. We say that $\varphi$ satisfies the Hasse norm principle if $K_\varphi$ does. We denote by $\gextk$ the set of all $G$-extensions of $k$. Our first theorem is then the following.

\begin{theorem} \label{thm:HNP_0}
	Let $n,r \in \ZZ$ with $n>1$ and $r \geq 0$. Let $Q$ be the smallest prime
	dividing $n$ and let $G = \ZZ/n\ZZ \oplus (\ZZ/Q\ZZ)^r$.
	Then $0\%$ of $G$-extensions of $k$ fail the Hasse norm principle.
\end{theorem}

Here, given a finite abelian group $G$, a number field $k$ and a property $P$, we say that $0\%$ of $G$-extensions
of $k$ satisfy $P$ if
$$\limsup_{B \to \infty} \frac{\#\{ \varphi\in\gextk: \, \Delta(\varphi) \leq B, \, \varphi
 \text{ satisfies } P\}}
{N(k,G,B)} = 0,$$
where
$$N(k,G,B) := \#\{ \varphi \in \gextk :\ \Delta(\varphi) \leq B \}$$
is the function which counts $G$-extensions of $k$ of discriminant at most $B$.

Our next result shows that Theorem \ref{thm:HNP_0} is non-trivial, namely that the set
of $G$-extensions which fail the Hasse norm principle is non-empty.

\begin{theorem} \label{thm:abelian}
	Let $G$ be a finite abelian group.
	Then there exists a $G$-extension of $k$ which fails the Hasse norm principle if and only if $G$ is non-cyclic.
\end{theorem}

This result was previously known only in some very special cases. For example, when $k=\QQ$ and $G$ is an elementary abelian $p$-group of even rank, Horie has shown the existence of infinitely many $G$-extensions of $\QQ$ that fail the Hasse norm principle \cite[Thm.~3]{Horie1993}.

We prove Theorem \ref{thm:abelian} using Shafarevich's resolution of the inverse
Galois problem for solvable groups, together with some cohomological calculations.
In particular, this also allows us to obtain the following
extension to solvable groups.

\begin{theorem} \label{thm:solvable}
	Let $G$ be a finite solvable group. Then there exists
	a $G$-extension of $k$ which fails the Hasse norm principle
	if and only if 
	$$\mathrm{H}^3(G,\ZZ) \neq 0.$$ 	
\end{theorem}

For a finite abelian group $G$, we have $\mathrm{H}^3(G,\ZZ) =0$ if and only if $G$ is cyclic (see, for example, Lemmas \ref{lem:exterior} and \ref{lem:exterior_calc}). In contrast, there are many finite non-abelian groups $G$ with $\mathrm{H}^3(G,\ZZ) = 0$; for example, groups for which every Sylow subgroup  is either cyclic or a generalised quaternion group (see \cite[Ch.~VI.9]{Bro82}, in particular Theorem 9.5 and Exercise 4). This observation allows us to obtain counter-examples to the statement of \cite[Cor.~3.2]{Gurak}, showing that the forward implication given there is false in general. For example, the quaternion group $\mathcal{Q}$ has $\mathrm{H}^3(\mathcal{Q},\ZZ) = 0$, yet is non-cyclic.

Returning to quantitative results, it turns out that there is an extreme dichotomy:~for all finite abelian non-cyclic groups $G$ other than those occurring in Theorem \ref{thm:HNP_0},
the Hasse norm principle fails for a positive proportion of $G$-extensions.

\begin{theorem} \label{thm:fail_HNP}
	Let $G$ be a non-trivial finite abelian group and let $Q$ be the smallest prime dividing $|G|$. 
	Assume that $G$ is not isomorphic to a group of the form
	$\ZZ/n\ZZ \oplus (\ZZ/Q\ZZ)^r$ for any $n$ divisible by $Q$ and $r \geq 0$.
	Then a positive proportion of $G$-extensions of $k$ fail the Hasse norm principle.
\end{theorem}
Here, given a finite abelian group $G$, a number field $k$ and a property $P$, we say that a positive proportion of $G$-extensions
of $k$ satisfy $P$ if
$$\liminf_{B \to \infty} \frac{\#\{\varphi\in\gextk: \, \Delta(\varphi) \leq B, \, \varphi \text{ satisfies } P\}}
{N(k,G,B)} > 0.$$

Note the distinction between our work and the work of Browning and Newton in \cite{BN15}.
There the authors fix an extension $K/\QQ$ and count those rational numbers of bounded
height that are norms everywhere locally but not globally, whereas in our work it is the extensions
themselves that are being counted.

\subsection{Norm one tori}
Our results on the Hasse norm principle may be interpreted via the arithmetic
of norm one tori. Given a finite field extension $K/k$, denote by
$$\Res_{K/k}^1 \Gm := \ker ( \Res_{K/k} \Gm \to \Gm )$$
the associated norm one torus. The failure of the Hasse norm principle for $K/k$ is completely controlled
by the Tate-Shafarevich group $\Sha(\Res_{K/k}^1 \Gm)$, in the sense that $K/k$ fails the Hasse norm
principle if and only if $\Sha(\Res_{K/k}^1 \Gm) \neq 0$ (see Section \ref{sec:tori}).
In particular, for $G$ as in Theorem \ref{thm:fail_HNP}, a positive proportion of the tori $\Res_{K/k}^1 \Gm$
have non-trivial Tate-Shafarevich group. Note that non-isomorphic Galois extensions give rise to non-isomorphic
norm one tori (see Proposition \ref{prop:tori}),
hence it makes sense to count such tori by the discriminant of the associated extension.

From the viewpoint of norm one tori, another interesting problem is to study the failure of
weak approximation for the tori $\Res_{K_\varphi/k}^1 \Gm$ as $\varphi$ varies over all \mbox{$G$-extensions} of $k$.
Recall that a variety $X$ over a number field $k$ is said to satisfy 
weak approximation if $X(k)$ is dense in $\prod_v X(k_v)$ for the product topology.
Here we have the following simple group-theoretic criterion, which gives a complete answer to this problem.

\begin{theorem} \label{thm:WA}
	Let $G$ be a non-trivial finite abelian group and let $Q$ be the smallest prime dividing $|G|$.
	Then as $\varphi$ varies over all $G$-extensions of $k$, the following hold.
	\begin{enumerate}
		\item If the $Q$-Sylow subgroup of $G$ is cyclic, 
		then a positive proportion of the tori $\Res_{K_\varphi/k}^1 \Gm$ satisfy weak approximation.
		\item If the $Q$-Sylow subgroup of $G$ is not cyclic, 
		then $0\%$ of the tori $\Res_{K_\varphi/k}^1 \Gm$ satisfy weak approximation.		
	\end{enumerate}
\end{theorem}

In particular, Theorem \ref{thm:WA} gives examples of $G$ for which $100\%$ of the 
tori $\Res_{K/k}^1 \Gm$ fail  weak approximation.
Results of this type were recently obtained in \cite{BBL15}, where the authors gave examples
of families of varieties for which $100\%$ of the members which are everywhere locally soluble
fail weak approximation.
Our set-up is slightly different to \cite{BBL15}, however, as there
the authors dealt with families of varieties ordered by height, where the family could contain many
isomorphic fibres, whereas here we only count each isomorphism
class over $k$ once. (To be more precise, we count each isomorphism class exactly $|\Aut(G)|$ times, but this is clearly equivalent.)

In \cite{BBL15} it was left open whether there exist families of varieties 
for which both a positive proportion of the everywhere locally soluble members satisfy weak approximation and a positive
proportion also fail weak approximation. 
We are able to give examples which satisfy this condition, at least in our setting.
These are obtained by combining Theorem \ref{thm:WA} with the following result.

\begin{theorem} \label{thm:satisfy_HNP}
	Let $G$ be a finite abelian non-cyclic group.
	Then as $\varphi$ varies over all $G$-extensions of $k$,
	the following hold.
	\begin{enumerate}
		\item A positive proportion of the $K_\varphi/k$ satisfy the Hasse norm principle.
		\item A positive proportion of the tori $\Res_{K_\varphi/k}^1 \Gm$ fail weak approximation.
	\end{enumerate}
\end{theorem}
Note that Theorem \ref{thm:satisfy_HNP}  implies the existence of some Galois extension $K/k$ with Galois group isomorphic to $G$
for which  $\Res_{K/k}^1 \Gm$ fails weak approximation;
something which is by no means obvious.

For any group as in Theorem \ref{thm:fail_HNP}, by Theorem \ref{thm:satisfy_HNP} we also obtain
examples of families of fields for which a positive proportion of the members satisfy the Hasse norm principle
and also a positive proportion fail the Hasse norm principle. This result should be compared with work of Bhargava 
\cite{Bha14c}, which shows that a positive proportion of cubic curves fail the Hasse principle, 
and a positive proportion also satisfy the Hasse principle.

\subsection{Counting with local conditions}
To prove our quantitative results, we require a theorem  concerning
counting abelian extensions with infinitely many local conditions imposed. 
This is Theorem \ref{thm:conditions} below, which is in some respects our main result. 
It states general conditions under
which one may obtain the existence of a  positive proportion of abelian extensions which satisfy infinitely many local conditions.

In order to state the result, we will need some more notation. 
Let $G$ be a finite abelian group, let $F$ be a field and $\bar{F}$ a separable closure of $F$. We define a \emph{sub-$G$-extension} of $F$ to be a continuous homomorphism $\Gal(\bar{F}/F)\to G$. A sub-$G$-extension corresponds to a pair $(L/F,\psi)$, where $L/F$ is a Galois extension inside $\bar{F}$ and $\psi$ is an injective homomorphism $\Gal(L/F)\to G$. 

For every place $v$ of the number field $k$, we fix an algebraic closure $\bar{k}_v$ and an embedding of $\bar{k}$ into $\bar{k}_v$. Hence, a sub-$G$-extension $\varphi$ of $k$ induces a sub-$G$-extension $\varphi_v$ of $k_v$ at every place $v$. For each place $v$ of $k$, let $\Lambda_v$ be a set of sub-$G$-extensions of $k_v$. For $\Lambda := (\Lambda_v)_v$ we study the function
\begin{equation} \label{def:counting_conditions}
N(k,G,\Lambda,B) := \#\left\{\varphi\in\gextk:\, \Delta(\varphi)\leq B, \, \varphi_v \in \Lambda_v \forall v \right\},
\end{equation}
which counts those $G$-extensions of $k$ of bounded discriminant which satisfy the local conditions imposed by $\Lambda$.

Let $Q$ be the smallest prime dividing $|G|$ and let $\phi_Q(G)$ be the number of elements of order $Q$ in $G$. Denote by $\mu_Q$ the group of $Q$th roots of unity in $\bar{k}$ and set
\begin{equation} \label{def:alpha}
	\alpha(G) := |G|(1 - Q^{-1}), \quad \nu(k,G) := \phi_Q(G)/[k(\mu_Q):k].
\end{equation}
Note that $\alpha(G), \nu(k,G)\in\NN$. Our main result is the following.

\begin{theorem} \label{thm:conditions} 
  	Let $G$ be a non-trivial finite abelian group and let $Q$ be the smallest prime dividing $|G|$. 
  	For each place $v$ of $k$, let $\Lambda_v$ be a set of sub-$G$-extensions of $k_v$ and let $\Lambda := (\Lambda_v)_v$.
  	Assume that for almost all places $v$,
        \begin{equation}\label{eqn:conditions}
\parbox{0.8\linewidth}{the set $\Lambda_v$ contains all sub-$G$-extensions $(K_v/k_v,\psi_v)$ for which the inertia group of $K_v/k_v$ has order dividing $Q$.}
        \end{equation}
  	Furthermore, suppose that there exists some $G$-extension $\varphi$ of $k$ such that $\varphi_v \in \Lambda_v$
  	for all $v$.
Then there is $\gamma>0$ and a real polynomial $P$ of degree $\nu(k,G)-1$ with positive leading coefficient, such that we have an asymptotic formula
\begin{equation*}
  N(k,G,\Lambda,B) = B^{1/\alpha(G)}P(\log B) + O(B^{1/\alpha(G)-\gamma}),\quad B\to\infty.
\end{equation*}
\end{theorem}
Note that Theorem \ref{thm:conditions} implies the existence of a positive proportion of $G$-extensions of $k$ satisfying the infinitely many local conditions imposed by $\Lambda$.

At first sight, the appearance of Condition \eqref{eqn:conditions} in Theorem \ref{thm:conditions} may seem surprising. Yet there is a simple heuristic
explanation for its presence. Namely, since we are ordering field extensions by their discriminants, the asymptotic behaviour
is controlled by the ``mildly'' ramified $G$-extensions. Condition \eqref{eqn:conditions} says that
provided one counts extensions with small inertia groups, one obtains enough $G$-extensions of small discriminant 
to guarantee a positive proportion of all \mbox{$G$-extensions}. Theorem \ref{thm:bicyclic} also
illustrates that the conclusion of Theorem \ref{thm:conditions} can fail if the local conditions $\Lambda_v$
do not satisfy \eqref{eqn:conditions}.

One of the main difficulties in the proof of Theorem \ref{thm:conditions} is that, in general,
the leading constant occurring in the asymptotic formula for $N(k,G,\Lambda,B)$ is an 
\emph{infinite} alternating sum of Euler products (see \eqref{eqn:constant}). 
We have to work quite hard to show that this infinite sum converges absolutely and 
that it is non-zero. In order to prove this non-vanishing,
it is crucial in Theorem \ref{thm:conditions} that we assume the existence of \emph{some} $G$-extension which 
realises all the local conditions imposed by $\Lambda$; there even exist
\emph{finite} sets of local conditions which are not realised by any $G$-extension of $k$, 
as first noticed by Wang \cite{Wan50} in his counter-example to Grunwald's ``theorem''.

\subsection{Proof strategy}
We now explain some of the methods used in this paper.
\subsubsection*{Counting with local conditions}
Our approach to counting abelian extensions of bounded discriminant is based upon the work of Wright \cite{Wri89},
who provided the definitive result on counting abelian extensions of bounded discriminant, building upon numerous special cases handled by other authors \cite{Cohn}, \cite{Uraz77}, \cite{Khush77}, \cite{Baily80}, \cite{Baily81}, \cite{Steckel} \cite{Taylor}, \cite{Zhang}, \cite{Mak85}. 
Like Wright, we proceed by studying the discriminant zeta function
\begin{equation}
\label{eq:zeta_function}
\sum_{\varphi}\frac{1}{\Delta(\varphi)^s}, \quad \re s \gg 1,
\end{equation}
where the sum is over suitable $G$-extensions of $k$. 
We relate the analytic properties of such zeta functions
to the original counting problem via standard Tauberian theorems (e.g.~\cite[Thm.~A.1]{CLT}).
Wright's method uses class field theory to translate the problem of counting $G$-extensions of $k$ into a problem
concerning the ideles of $k$.
 Note that Wright considered the slightly different problem of counting Galois extensions $K/k$ inside $\bar{k}$ with Galois group isomorphic to $G$, without specifying the isomorphism. We include the choice of isomorphism as part of the data as it is required for our applications, leads to a more natural framework, and also simplifies some parts of the proof. 

There are a few differences between our approach and Wright's.
First, we replace some of the ad hoc steps in his paper
by more conceptual methods, namely harmonic analysis on the idele class group. 
This new approach leads to a more powerful and versatile framework. This is for more than
just aesthetic purposes; it is exactly the tools from harmonic analysis which allow us to 
count number fields with infinitely many local conditions, via a version of the Poisson
summation formula. 
This approach is inspired, at least on a philosophical level, 
by the harmonic analysis approach to Manin's conjecture for toric varieties,
as pioneered by Batyrev and Tschinkel \cite{BT98}.
Next, at numerous points Wright uses the ``fact'' that
the map
$$k^*/k^{*n} \to \Adele_k^* /\Adele_k^{*n},$$
is injective for all $n \in \NN$; see for instance the  claim just before \cite[Lem.~3.1]{Wri89}
and the discussion following \cite[Prop.~4.1]{Wri89}.
This is not true in general, as first noticed by Wang \cite{Wan50} (see \cite[Ch.~IX.1]{NSW00} for a modern account).
Our harmonic analysis approach highlights
where these mistakes occur and moreover shows why they do not affect Wright's main result (see e.g.~the proof of Proposition \ref{prop:Poisson}).

\subsubsection*{Hasse norm principle and weak approximation}
In order to apply our counting results to the study of the Hasse norm principle,
we use a theorem of Tate which explicitly calculates
the Tate-Shafarevich group of the associated norm one torus (see Theorem \ref{thm:Tate}). 
Whilst this gives a complete criterion to check whether a given Galois extension
satisfies the Hasse norm principle, part of the difficulty with combining this with our counting results
is that, in general, this criterion is not a simple collection of independent local conditions,
i.e.,~different places may interfere with each other. However, in Section \ref{sec:HPandWA}
we use Tate's theorem to obtain both necessary and sufficient collections of 
independent local conditions for the Hasse norm
principle to hold, and it is these conditions that turn out to be susceptible to our counting methods. 
Note that these difficulties do not arise for the question of weak approximation: it turns out that
the problem of whether $\Res_{K/k}^1 \Gm$ satisfies weak approximation \emph{is} completely
controlled by independent local conditions when $K/k$ is abelian (see Lemma \ref{lem:WA}).

\subsection{Concluding remarks}
We hope that the tools and results developed in this paper for counting abelian extensions with infinitely many local conditions imposed may be of independent interest. Indeed, for many applications it is often necessary to count extensions with imposed local conditions,
e.g.~counting only those extensions for which certain primes have some given splitting behaviour. 
Analogous results have recently been obtained for extensions of degrees $d \in \{2,3,4,5\}$
with Galois group $S_d$ and fairly general local conditions by Bhargava \cite{Bha14d}.
Wright also has a result \cite[Thm.~1.4]{Wri89} on
counting abelian extensions with a \emph{single} local condition and there is, moreover, work of Wood \cite{Woo10}
on counting abelian extensions of \emph{bounded conductor} with finitely many local conditions imposed.
However, the authors know of no previously existing work where general abelian extensions have been successfully 
counted with \emph{infinitely} many non-trivial local conditions.

We finish by discussing possible further directions. First, as in Wright's paper, 
our methods should extend to global fields, providing that
one avoids suitable characteristics as Wright does. Secondly, in this paper we have focused on counting abelian extensions by \emph{discriminant}, in order to
put the problem into a more general framework. Another natural problem would be to 
instead count abelian extensions of bounded \emph{conductor}, by combining our harmonic
analysis methods with the approach of Wood \cite{Woo10}. 
Thirdly, one could try to improve Theorem~\ref{thm:HNP_0} from a zero density result to a precise asymptotic formula;
even the biquadratic case appears to be non-trivial.
Finally, there is a wealth of literature on counting number fields of bounded discriminant.
It would be interesting to see if these tools could be adapted to count failures of the
Hasse norm principle for some non-abelian extensions.

\subsection{Outline of the paper}
Sections \ref{sec:disc} and \ref{sec:harmonic} are devoted to setting up our general framework for  counting abelian extensions. In particular, we use class field theory to translate the problem to the study of certain conductor series, which we explain how to attack using tools from harmonic analysis in Section \ref{sec:harmonic}.

Section \ref{sec:conditions} is the technical heart of the paper. 
Here we use the general tools developed in the previous two sections to prove Theorem \ref{thm:conditions}, which is in some respects our main theorem. Section \ref{sec:bicyclic} contains some applications of Theorem \ref{thm:conditions} which will be used in our later proofs.

In Section \ref{sec:HPandWA}, we finish by proving the remaining theorems stated in the introduction. These are obtained by combining various algebraic and cohomological calculations with suitable applications of Theorem \ref{thm:conditions} and the results from Section \ref{sec:bicyclic}.

\subsection{Notation and conventions} \label{sec:notation}

We fix a number field $k$ throughout the paper and use the following notation.

\begin{longtable}{ll}
$\Adele^*$ & the ideles of $k$\\
$\OO_k$ &the ring of integers of $k$\\
$\Omega_k$ & the set of all places of $k$\\
$\Omega_\infty$ & the set of all archimedean places of $k$\\
$\Omega_L \text{ or } \Omega(L)$ & the set of all places of a finite extension $L$ of $k$\\
$S$ & a finite set of places of $k$\\
$\OO_S$ &the $S$-integers of $k$\\
$\Adele^*_S$ \text{or} $\Adele^*(S)$ & the $S$-ideles of $k$, $\Adele^*_S:=\prod_{v\in S}{k_v^*}\prod_{v\notin S}{\OO_v^*}$\\
$v$ & a place of $k$\\
$k_v$ & the completion of $k$ at $v$\\
$\OO_v$ & the ring of integers of $k_v$. For $v\mid \infty$, by convention $\OO_v := k_v$\\
$\pi_v$ & a uniformiser at a finite place $v$\\
$q_v$ & the cardinality of the residue field at a finite place $v$\\
$K$ & a Galois extension of $k$\\
$K_v$ & the completion of $K$ at some choice of place above $v$\\
$D_v$ & the Galois group of $K_v/k_v$\\
$I_v$ & the inertia group of $K_v/k_v$
\end{longtable}

\smallskip

\noindent For a Dedekind domain $R$, we use the following notation.

\begin{longtable}{ll}
$I(R)$ & the group of non-zero fractional ideals of $R$\\
$P(R)$ & the subgroup of principal ideals in $I(R)$\\
$\Pic(R)$ & the ideal class group of $R$, $\Pic(R):=I(R)/P(R)$
\end{longtable}

\smallskip

\noindent All locally compact abelian groups in this paper are assumed to be Hausdorff.
For locally compact abelian groups $A$ and $B$, we use the following notation.

\begin{longtable}{ll}
$\Hom(A,B)$ & the group of \emph{continuous} homomorphisms from $A$ to $B$, \\
& equipped with the
compact-open topology\\
$S^1$ & the unit circle \\
$\mu_n$ & the group of $n$th roots of unity\\
$\dual{A}$ & the Pontryagin dual of $A$, $\dual{A} := \Hom(A, S^1)$\\
$\QZdual{A}$ & the $\QQ/\ZZ$-dual of $A$, $\QZdual{A} := \Hom(A, \QQ/\ZZ)$\\
$\pair{\cdot}{\cdot}$ & the natural pairing $A\times \dual{A} \to S^1$\\
\end{longtable}

\noindent If $B$ is finite or $B = S^1$, then $\Hom(A,B)$ is itself a locally compact abelian group
(this is a special case of \cite[Cor., p.~377]{Mos67}).
\smallskip

\noindent 
If $B$ is a closed subgroup of $A$, then we identify $\dual{(A/B)}$ with the subgroup $B^\perp$ of characters of
$A$ that are trivial on $B$. We view $k^*$ and $k_v^*$ as closed subgroups of the locally compact abelian group $\Adele^*$. The induced topology on $k^*$ is discrete. For a finite abelian group $G$ and $x\in k^* \otimes G$, we write $x_v$ for its image under the natural map $k^* \otimes G\to k_v^* \otimes G$ induced by the natural inclusion $k^*\to k_v^*$; note that this former map is not injective in general.

\smallskip

\noindent 
For a finite abelian group $G$ endowed with the discrete topology, we use the following notation.

\begin{longtable}{ll}
$|G|$ & the order of $G$\\
$\exp(G)$ & the exponent of $G$\\
$Q$ & the smallest prime dividing $|G|$\\
$\phi_Q(G)$ & the number of elements of order $Q$ in $G$\\
$\beta_G:=\log_Q(\phi_Q(G) + 1)$ & \\
$\alpha(G):=|G|(1-Q^{-1})$ & \\
 $\nu(k,G) := \phi_Q(G)/[k(\mu_Q):k]$ \\
\end{longtable}

\smallskip

\noindent 
Every finite abelian group $G$ may be written uniquely as 
\begin{equation} \label{eqn:n_j}
	G = \ZZ/n_1\ZZ \oplus \cdots \oplus \ZZ/n_l\ZZ,
\end{equation}
with $n_j > 1$ and $n_{j+1} \mid n_j$ for $1\leq j< l$.
We reserve the variables $n_1,\ldots,n_l$ for the above notation throughout the paper. 

\noindent 
For a commutative group scheme $X$ over a number field $k$, we denote by 
\begin{equation} \label{def:Sha}
	\Sha(k, X) := \ker \left( \mathrm{H}^1(k,X)  \to \prod_v \mathrm{H}^1(k_v,X) \right)
\end{equation}
the Tate-Shafarevich group  of $X$. We shall often omit $k$ if it is clear.
The non-zero elements of $\Sha(k,X)$ classify those $X$-torsors which fail the Hasse principle.

Given a set $X$, we denote by $\Const_X: X \to \{1\}$ the constant function with value $1$
(we shall often omit the subscript if it is clear from the context).

Given $a,b\in \ZZ$,  we write $(a,b)$ for their greatest common divisor.

\subsection{Acknowledgements}
Part of this work was completed whilst attending the \emph{Rational Points 2015} workshop -- we thank Michael Stoll for providing such a stimulating working environment. We thank Jeremy Rickard and Norbert Hoffmann for useful discussions, and
are grateful to Manjul Bhargava and Melanie Matchett Wood for their interest in our paper. We also thank Jean-Louis Colliot-Th\'{e}l\`{e}ne for useful comments and references, and the referees for helpful remarks.

\section{From discriminant series to conductor series} \label{sec:disc}
Let $k$ be a number field and let $G$ be a non-trivial finite abelian group. We begin by describing our set-up for counting $G$-extensions of $k$ and the transition to the study of conductor series. This closely follows \cite[\S2]{Wri89}, though unlike Wright we do not choose a presentation of $G$.

\subsection{Discriminant series}

Let $f$ be any complex-valued function on the set of sub-$G$-extensions of $k$. Consider the (formal, for now) Dirichlet series
\begin{equation}\label{eq:discriminant_series}
  D_{G,f}(s):=\sum_{\varphi\in \gextk}\frac{f(\varphi)}{\Delta(\varphi)^s}.
\end{equation}
Since $G$ is abelian, global class field theory allows us to identify $G$-extensions of $k$ with surjective continuous homomorphisms $\Adele^*/k^*\to G$, where $G$ is endowed with the discrete topology. The conductor-discriminant formula gives
\begin{equation}\label{eq:conductor_discriminant}
  D_{G,f}(s) = \sum_{\varphi\in\gextk}f(\varphi)\prod_{\psi\in\dual{(\Adele^*/\ker\varphi)}}\Phi(\psi)^{-s},
\end{equation}
where $\Phi(\psi)$ is the reciprocal of the idelic norm of the conductor of $\psi$. 
To deal with the rather unpleasant surjectivity condition, we introduce certain conductor series.

\subsection{Conductor series}
Let $X$ be one of the following locally compact abelian groups: the group $\Adele^*$ of ideles, the idele class group  $\Adele^*/k^*$, the local multiplicative group $k_v^*$ at the place $v$ of $k$, or the local unit group $\OO_v^*$ at $v$. In each case, the conductor of a character $\psi\in\dual X$ of finite order is defined as the conductor of its kernel. Let us denote by $\Phi(\psi)$ the reciprocal of the (idelic or $v$-adic) norm of this conductor.

Let $H$ be a finite abelian group endowed with the discrete topology, and let $\chi \in \Hom(X,H)$. The kernel of $\chi$ is a finite-index open subgroup of $X$. By Pontryagin duality, we obtain from $\chi$ a canonical surjective homomorphism
\begin{equation*}
  \widehat{\chi} : \dual{H} \to \dual{(X/\ker\chi)}.
\end{equation*}
We define
\begin{equation}\label{eq:def_Phi_H}
  \Phi_H(\chi) := \prod_{h\in \dual{H}}\Phi(\widehat{\chi}(h)) = \prod_{\psi\in\dual{(X/\ker\chi)}}\Phi(\psi)^{|H|/|X/\ker\chi|}.
\end{equation}
Note that there is a natural inclusion $\Hom(\Adele^*/k^*, H) \subseteq \Hom(\Adele^*, H)$, and for any $\chi \in \Hom(\Adele^*/k^*, H)$, the two corresponding definitions \eqref{eq:def_Phi_H} of $\Phi_H(\chi)$ coincide. In the case where $X = \Adele^*$, we have a restricted direct product
\begin{equation} \label{def:restricted_product}
\Hom(\Adele^*, H) = \lim_{S\to\Omega_k} \Hom(\Adele^*, H)(S)
\end{equation}
(see e.g. \cite[Thm.~1.7]{Mor05}), where
$$\Hom(\Adele^*, H)(S) := 
\prod_{v \in S} \Hom(k_v^*, H) \times \prod_{v \not \in S} \Hom(k_v^*/\OO_v^*, H).$$
Note that $\Hom(k_v^*,H)$ is finite and discrete, and coincides with the group of abstract group homomorphisms $k_v^*\to H$. Moreover, $\Hom(k_v^*/\OO_v^*,H)$ may be identified with those homomorphisms $k_v^*\to H$ which are trivial on $\OO_v^*$. Here for archimedean $v$, by convention we set $\OO_v^* = k_v^*$. For non-archimedean $v$, a choice of uniformiser in $k_v$ yields an isomorphism
\begin{equation}\label{eq:local_character_direct_sum}
  \Hom(k_v^*,H)\cong\Hom(k_v^*/\OO_v^*,H)\oplus\Hom(\OO_v^*,H).
\end{equation}
For $\chi_1 \in\Hom(k_v^*/\OO_v^*, H)$ and $\chi_2 \in \Hom(k_v^*, H)$, we have $\Phi_H(\chi_1\chi_2) = \Phi_H(\chi_2)$. 
Let $\chi = (\chi_v)_v \in \Hom(\Adele^*, H)$, where $\chi_v \in \Hom(k_v^*, H)$. Then $\Phi_H(\chi_v) = 1$ for all but finitely many $v$, and
\begin{equation*}
  \Phi_H(\chi) = \prod_{v}\Phi_H(\chi_v).
\end{equation*}
We are now ready to define the conductor series
\begin{equation}\label{eq:conductor_series}
  F_{H,f}(s) := \sum_{\chi \in \Hom(\Adele^*/k^*, H)}\frac{f(\chi)}{\Phi_H(\chi)^{s}}.
\end{equation}

\begin{lemma}\label{lem:conductor_to_discriminant}
  For any subgroup $H$ of $G$, we have
  \begin{equation*}
    F_{H,f}(s) = \sum_{J\subseteq H}D_{J,f}(s|H|/|J|),
  \end{equation*}
  where the sum runs over all subgroups $J$ of $H$. The series $D_{J,f}(s)$ is defined as in \eqref{eq:discriminant_series}, but with $G$ replaced by $J$.
\end{lemma}

\begin{proof}
If the image of $\chi\in\Hom(\Adele^*/k^*,H)$ is a subgroup $J$ of $H$, then
\begin{equation*}
  \Phi_H(\chi) = \prod_{\psi\in \dual{(\Adele^*/\ker\chi)}}\Phi(\psi)^{|H|/|J|}.
\end{equation*}
Now sort the homomorphisms $\chi\in\Hom(\Adele^*/k^*,H)$ by their images.
\end{proof}

\subsection{M\"obius inversion}

Let $\mu$ be the M\"obius function on isomorphism classes of finite abelian groups. That is, $\mu(G) = 0$ if $G$ has a cyclic subgroup of order $p^n$, with $p$ a prime and $n\geq 2$, $\mu(G_1\times G_2) = \mu(G_1)\mu(G_2)$ if $G_1$ and $G_2$ have coprime order, and $\mu((\ZZ/p\ZZ)^n) = (-1)^np^{n(n-1)/2}$ for a prime $p$ and $n \in \NN$.

\begin{lemma}\label{lem:discriminant_to_conductor}
  We have
  \begin{equation*}
    D_{G,f}(s) = \sum_{H\subseteq G}\mu(G/H)F_{H,f}(s|G|/|H|),
  \end{equation*}
  where the sum runs over all subgroups $H$ of $G$.
\end{lemma}

\begin{proof}
This is Delsarte's inversion formula \cite{Delsarte}, applied as in \cite[\S 2]{Wri89}.
\end{proof}

\section{Harmonic analysis} \label{sec:harmonic}
In the previous section, we reduced the study of our problem to the study of various conductor series. We shall tackle these using tools from harmonic analysis. The main result of this section is Proposition \ref{prop:Poisson}, which is a version of Poisson summation. Throughout this section we work over a number field $k$ and with a finite abelian group $G$ endowed with the discrete topology; note that in later sections we shall also apply the following results to subgroups $H$ of a given discrete finite abelian group $G$.

We will assume that the reader is familiar with the basic definitions and tools from abstract harmonic analysis on locally compact abelian groups. A detailed treatment can be found in \cite[Ch.~2]{Bou67}. For a general overview, see \cite[\S1.1]{Mor05}.

\subsection{Topological considerations} \label{subsec:top}
We begin our analysis with some topological considerations. In this paper, we will need to consider certain tensor products of topological groups. This theory can be quite complicated in general; see for example \cite{HS07} and \cite{Mos67}. However, we will only require this theory in some special cases where everything is well-behaved, which we now explain.

Our first result may be viewed as a generalisation of \cite[Cor., p.~386]{Mos67}.

\begin{lemma} \label{lem:quotient}
	Let $X$ be a locally compact abelian group. Let $n \in \NN$ and assume that
	$nX$ is a closed subgroup of $X$. Then there is a canonical isomorphism
	$$\dual{\Hom(X, \mu_n)} \cong X/nX$$
	of locally compact abelian groups, where $X/nX$ is equipped with the quotient
	topology.
\end{lemma}
\begin{proof}
	As $nX \subset X$ is closed, we have 
	$$(nX)^\perp \cong \dual{(X/nX)}.$$
	However, by definition we have
	$$\Hom(X, \mu_n) = \{ \chi \in \dual{X}: \chi^n = 1\} = (nX)^\perp.$$
	Applying Pontryagin duality, we obtain the result.
\end{proof}

For many locally compact abelian groups $X$, the subgroup $nX \subset X$ is closed 
for all $n\in \NN$.
For example, this is the case if $X$ is compact or discrete. This property is closed under arbitrary direct products and direct sums (but not under restricted direct products, in general). Moreover, an elementary argument shows that $nX$ is a closed subgroup for all $n\in \NN$ when $X=\Adele^*$ or $X=\Adele^*/k^*$. These cases will be sufficient for all the applications required in this paper.

\begin{lemma} \label{lem:tensor}
	Let $X$ be a locally compact abelian group and let $G$ be a finite abelian group.
	Assume that $nX$ is a closed subgroup of $X$ for all $n \mid |G|$. Then there is 
	a canonical isomorphism 
	$$\dual{\Hom(X, G)} \cong X \otimes \dual{G}$$
	of abelian groups.
\end{lemma}
Here $X \otimes \dual{G}$ denotes the usual tensor product of $X$ with $\dual{G}$, considered as abelian groups (i.e. forgetting the topology on $X$).
\begin{proof}
	Consider the natural homomorphism
	$$X \otimes \dual{G} \to \dual{\Hom(X, G)}, \quad 
	x \otimes \chi \mapsto \left((\phi: X \to G) \mapsto \chi(\phi(x)) \right)$$
	(a simple calculation shows that the induced homomorphism $\Hom(X,G) \to S^1$ is indeed
	continuous).
	To show that this is an isomorphism, it suffices to treat the case where $G = \mu_n$.
	Here we obtain the following commutative diagram
	$$
	\xymatrix{
	X \otimes \ZZ/n\ZZ \ar[r] \ar[dr] & \dual{\Hom(X, \mu_n)} \\
	& X/nX. \ar[u] \\
	}$$
	The diagonal arrow is the usual isomorphism of abelian groups, and the vertical
	arrow is the isomorphism of Lemma \ref{lem:quotient}. This proves the result.
\end{proof}

Let $X$ be as in Lemma \ref{lem:tensor}. Then Lemma \ref{lem:tensor} canonically identifies the Pontryagin dual of $\Hom(X, G)$ with $X \otimes \dual{G}$, \emph{as an abelian group}. Since $\dual{\Hom(X, G)}$ carries a topology, we may use this isomorphism to equip $X \otimes \dual{G}$ with the structure of a locally compact abelian group. Lemma \ref{lem:quotient} gives a very simple way to visualise the topology on $X \otimes \dual{G}$, on choosing a presentation of $G$. For example, if $X$ is discrete then we see that $X \otimes \dual{G}$ is also discrete.

Henceforth, for $X$ as in Lemma \ref{lem:tensor}, \emph{we identify $X \otimes \dual{G}$ with the Pontryagin dual of $\Hom(X, G)$, viewed as locally compact abelian groups}. We denote by 
$$ \langle \cdot , \cdot \rangle: \Hom(X, G) \times (X \otimes \dual{G}) \to S^1$$
the pairing induced by Pontryagin duality.

\subsection{Haar Measures} \label{sec:measures}
We now return to the study of $\Hom{(\Adele^*,G)}$,
beginning with a choice of normalisation for its Haar measure.
For each place $v$, we equip $\Hom(k_v^*,G)$ with the unique Haar measure $\mathrm{d}\chi_v$ such
that $$\vol(\Hom(k_v^*/\OO_v^*,G)) = 1.$$
For archimedean $v$ this is simply the counting measure (as $\OO_v^*=k_v^*$ by convention),
whilst for non-archimedean $v$ the measure $\mathrm{d}\chi_v$ is a factor of
$$ \frac{1}{ \#\Hom(k_v^*/\OO_v^* , G)} = \frac{1}{ |G|} $$
times the counting measure. Recalling the restricted product description \eqref{def:restricted_product}
of $\Hom(\Adele^* , G)$, one easily sees that the product of these 
Haar measures converges to give a Haar measure $\mathrm{d}\chi$ on $\Hom(\Adele^* , G)$.

\subsection{Fourier transforms}\label{subsec:fourier_transforms}
Throughout this section, for each place $v$ we let \mbox{$f_v: \Hom(k_v^* , G)\to \CC$} be 
functions  which take only the value $1$ on $\Hom(k_v^*/\OO_v^* , G)$ for all but finitely many 
$v$. The product $f=\prod_v f_v: \Hom(\Adele^* , G)\to \CC$ is then a well-defined continuous function. 
For $s\in\CC$, $x_v \in k_v^* \otimes \dual{G}$ and $x=(x_v)_v \in \Adele^*\otimes \dual{G}$, we denote by
$$\widehat{f}_v(x_v;s) 
= \int_{\chi_v \in \Hom(k_v^* , G)} \frac{f_v(\chi_v)\pair{\chi_v}{x_v}}{\Phi_G(\chi_v)^{s}}  \mathrm{d}\chi_v,
\,\, \widehat{f}(x;s) = \int_{\chi \in \Hom(\Adele^* , G)} \frac{f(\chi)\pair{\chi}{x}}{\Phi_G(\chi)^{s}}  \mathrm{d}\chi$$
the local and global Fourier transforms of the function $f \Phi_G^{-s}$, respectively, if they exist. 
Here, as explained in Section \ref{subsec:top}, we are identifying $\Adele^*\otimes \dual{G}$ and 
$k_v^* \otimes \dual{G}$ with the Pontryagin duals of $\Hom(\Adele^* , G)$ and $\Hom(k_v^* , G)$, respectively,
and denote by $\langle \cdot, \cdot \rangle$ the pairing induced by Pontryagin duality.
We first calculate the local Fourier transforms at almost all places. Let $v$ be non-archimedean.
Note that, by our choice of measures, we have
$$\widehat{f}_v(x_v;s) 
= \frac{1}{|G|}\sum_{\chi_v \in \Hom(k_v^* , G)}\frac{f_v(\chi_v)\pair{\chi_v}{x_v}}{\Phi_G(\chi_v)^{s}}.$$
Recalling \eqref{eq:local_character_direct_sum}, for 
$x_v \in k_v^*\otimes \dual{G}$ and $\chi_v \in \Hom(\OO_v^*, G)$ we define the ``average''
\begin{equation*}
  \tau_{f_v}(\chi_v, x_v) := \frac{1}{|G|}\sum_{\psi_v \in \Hom(k_v^*/\OO_v^*, G)} f_v(\chi_v\psi_v)\pair{\psi_v}{x_v}.
\end{equation*}
\begin{lemma}\label{lem:local}
  Let $v$ be a non-archimedean place of $k$ such that $(\exp(G),q_v)=1$. 
  Then for $x_v \in k_v^* \otimes \dual{G}$ we have
  \begin{equation*}
    \widehat{f}_v(x_v;s) = \sum_{m\mid (\exp(G),q_v-1)}\left(\sum_{\substack{\chi_v\in\Hom(\OO_v^*, G)\\ \ker\chi_v=\OO_v^{*m}}}\pair{\chi_v}{x_v}\tau_{f_v}(\chi_v,x_v)\right)q_v^{-|G|(1-1/m)s}.
\end{equation*}
\end{lemma}

\begin{proof}
  This is inspired by the proof of \cite[Prop.~4.1]{Wri89}.
  From \eqref{eq:local_character_direct_sum}, we obtain
  \begin{equation*}
    \widehat{f}_v(x_v;s) = \sum_{\chi_v\in\Hom(\OO_v^*, G)}\frac{\pair{\chi_v}{x_v}}{\Phi_G(\chi_v)^s}\tau_{f_v}(\chi_v,x_v).
  \end{equation*}
Let $\chi_v\in\Hom(\OO_v^*, G)$. Since $q_v$ is coprime to $\exp(G)$, we get $1+\pi_v\OO_v\subseteq \ker\chi_v$. Hence, $\chi_v$ factors through the units of the residue field, so $\ker\chi_v = \OO_v^{*m}$ for some divisor $m$ of $(\exp(G),q_v-1)$.
Hence, sorting the $\chi_v$ by their kernels,
\begin{equation*}
  \widehat{f}_v(x_v;s) = \sum_{m\mid (\exp(G),q_v-1)}\sum_{\substack{\chi_v\in\Hom(\OO_v^*, G)\\ \ker\chi_v=\OO_v^{*m}}}\frac{\pair{\chi_v}{x_v}}{\Phi_G(\chi_v)^s}\tau_{f_v}(\chi_v,x_v).
\end{equation*}
For $\chi_v \in \Hom(\OO_v^*, G)$ with $\ker\chi_v=\OO_v^{*m}$, we have
\begin{equation*}
  \Phi_G(\chi_v) = \prod_{\chi \in \dual{(\OO_v^*/\OO_v^{*m})}}\Phi(\chi)^{|G|/m} = q_v^{|G|\frac{m-1}{m}}.  
\end{equation*}
Indeed, the trivial character $\Const$ has $\Phi(\Const) = 1$, and the other $m-1$ characters $\chi$ have conductor $\pi_v$, so $\Phi(\chi) = q_v$.
\end{proof}

We now establish rapid decay for the local Fourier transforms, which will be used to guarantee
the convergence of the Poisson summation formula. Furthermore, we show that if $f$ is
suitably nice, then the local Fourier transforms are often zero. In the next lemma, we let $\alpha(G)$
be as in \eqref{def:alpha}.

\begin{lemma} \label{lem:zero}
	Let $v$ be a non-archimedean place of $k$ such that $(\exp(G),q_v)=1$ and such that $f_v(\Hom(k_v^*/\OO_v^* ,G))=1$.
	Let $\gamma = \max\{|f_v(\chi_v)|\}$ be the maximum taken over all $\chi_v \in \Hom(k_v^* , G)$. Then
	$$\widehat{f}_v(x_v;s) =
	\begin{cases}
		1 + O_\gamma(q_v^{-\alpha(G) s}), & x_v\in \OO_v^*\otimes \dual{G}, \\
		O_\gamma(q_v^{-\alpha(G) s}), & x_v \not \in \OO_v^*\otimes \dual{G}.
	\end{cases}
	$$
	Moreover, if $f_v$ is $\Hom(k_v^*/\OO_v^* , G)$-invariant and $x_v \not \in \OO_v^* \otimes \dual{G}$,
	then $$\widehat{f}_v(x_v;s) = 0.$$
\end{lemma}
\begin{proof}
	For $x_v\in \OO_v^*\otimes \dual{G}$, the result follows easily from the calculation of the lower
	order terms in Lemma \ref{lem:local}. Henceforth, let $x_v \notin \OO_v^*\otimes \dual{G}$. Then by character orthogonality
	we have $$\tau_{f_v}(\Const_v, x_v) = 0,$$
	where $\Const_v: \OO_v^* \to G$ denotes the trivial homomorphism.
	The first part of the lemma now follows easily from Lemma \ref{lem:local}. Under the additional assumption
	that $f_v$ is $\Hom(k_v^*/\OO_v^* , G)$-invariant, we similarly obtain
	$$\tau_{f_v}(\chi_v, x_v) = 0$$
	for all $\chi_v \in \Hom(\OO_v^*, G)$. The result then follows from Lemma \ref{lem:local}.
\end{proof}

We now show the existence of the global Fourier transform and an Euler product formula under
suitable assumptions.

\begin{lemma} \label{lem:Euler_product}
	Assume that there exists some $\gamma>0$ such that $|f_v(\chi_v)| \leq \gamma$ for all places $v$ 
	and all $\chi_v \in \Hom(k_v^*,G)$. Then, in the half plane $\re s > 1/\alpha(G)$, the global Fourier
	transform $\widehat{f}(\cdot;s)$ of $f\Phi_G^{-s}$ defines a holomorphic function
and
	$$\widehat{f}((x_v)_v;s) = \prod_v \widehat{f}_v(x_v;s), \qquad \re s > 1/\alpha(G).$$
\end{lemma}
\begin{proof}
	It follows easily from Lemma \ref{lem:zero} that the above infinite product is absolutely
	convergent for $\re s > 1/\alpha(G)$. The result then follows from a
	standard argument (see e.g. \cite[Lem.~1.4]{Mor05}).
\end{proof}


\subsection{Poisson summation}

The main result of this section is a version of Poisson summation
applied to the conductor series. We begin with an algebraic lemma.
For $x$ in $k^*\otimes \dual{G}$, we denote its image in $k_v^*\otimes \dual{G}$ by $x_v$.

\begin{lemma} \label{lem:U}
	Let $S$ be a finite set of places of $k$ and let
	\begin{equation} \label{def:U_S(G)}
		\mathcal{U}_S(G) = \{x \in k^*\otimes \dual{G} : x_v \in \OO_v^*\otimes \dual{G} \ \forall v \notin S\}.
	\end{equation}
	Then there is an exact sequence of abelian groups
	$$ 1 \to \OO_S^* \otimes \dual{G} \to \mathcal{U}_S(G) \to \Tor( \Pic \OO_S, \dual{G}) \to 1.$$
	In particular, $\mathcal{U}_S(G)$ is finite.
\end{lemma}
\begin{proof}
	 Denote by $I(\OO_S)$ the group of non-zero fractional ideals of $\OO_S$, and denote by $P(\OO_S)$ the subgroup generated by the principal ideals. Note that $\Tor(I(\OO_S),G)=\Tor(P(\OO_S),G)=0$.	Therefore, we have the following commutative diagram of abelian groups
		$$
	\xymatrix{
	0 \ar[r] & \Adele^*_S \otimes \dual{G}
	\ar[r] & \Adele^*\otimes\dual{G}\ar[r]& I(\OO_S) \otimes \dual{G} \ar[r] & 0 \\
	0 \ar[r] &\OO_S^* \otimes \dual{G}\ar[r] \ar[u] & k^*\otimes\dual{G}  \ar[r] \ar[u] 
	& P(\OO_S) \otimes \dual{G} \ar[u]\ar[r] & 0 \\
	 &&  & \Tor( \Pic \OO_S, \dual{G})\ar[u] &  \\
	 &&  & 0 \ar[u] & \\
	}$$
with exact rows and columns.  The result then follows from a simple diagram chase.
\end{proof}

The following lemma will be used to show convergence of the Poisson sum.

\begin{lemma} \label{lem:k*_sum}
	Let $g_v: k_v^*\otimes \dual{G} \to \CC$ for each place $v$ of $k$. Assume that there exists some $\sigma > 1$
	such that 
	$$g_v(x_v) =
	\begin{cases}
		1 + O(q_v^{-\sigma}), & x_v\in \OO_v^*\otimes \dual{G}, \\
		O(q_v^{-\sigma}), & x_v \not \in \OO_v^*\otimes \dual{G}.
	\end{cases}
	$$
	Then $g = \prod_v g_v: \Adele^*\otimes \dual{G} \to \CC$ is well defined and the sum
	$$\sum_{x \in k^*\otimes \dual{G}} g(x)$$
	is absolutely convergent.
\end{lemma}
\begin{proof}
	That $g$ is well defined is clear. Choose some finite set of places $S$ such that $\Pic \OO_S = 0$.
	Let $I(\OO_S)$ denote the group of non-zero fractional ideals of $\OO_S$.
	Then, since $\Tor(I(\OO_S),\dual{G})=0$, we have the following short exact sequence
	$$ 1 \to \OO_S^* \otimes \dual{G} \to k^*\otimes\dual{G} \to I(\OO_S) \otimes \dual{G} \to 1.$$
	Let $\mathcal{R}$ be a system of representatives for the cosets of $\OO_S^*\otimes \dual{G}$ in $k^*\otimes \dual{G}$. 
	This allows us to write
	$$
		\sum_{x \in k^*\otimes \dual{G}} |g(x)| = \sum_{u \in \OO_S^* \otimes \dual{G}}\sum_{r \in \mathcal{R}} |g(ur)|.
	$$
	It now follows from Dirichlet's $S$-unit theorem and our assumptions on $g$ that there
	exists some constant $\gamma> 0$ and some $0 < \varepsilon < \sigma -1$ such that
	\begin{equation} \label{eqn:gamma}
	\sum_{x \in k^*\otimes\dual{G}} |g(x)| \ll 
	\sum_{r \in \mathcal{R}} 
	\prod_{\substack{v\notin S \\ r_v \in\OO_v^*\otimes \dual{G}}}\left( 1 + \frac{\gamma}{q_v^{\sigma}}\right)
	\prod_{\substack{v\notin S \\ r_v\notin\OO_v^*\otimes \dual{G}}}\frac{\gamma}{q_v^{\sigma}}
	\ll
	\sum_{r \in \mathcal{R}} \prod_{\substack{v\notin S \\ r_v\notin\OO_v^*\otimes \dual{G}}}\frac{1}{q_v^{\sigma-\varepsilon}}.
	\end{equation}
	To proceed, choose a presentation 
	$$\dual{G} \cong \ZZ/n_1\ZZ \oplus \cdots \oplus \ZZ/n_l\ZZ.$$
	This  induces an isomorphism
	\begin{equation*}
	  I(\OO_S) \otimes \dual{G} \cong \bigoplus_{j=1}^l I(\OO_S)/I(\OO_S)^{n_j}. 
	\end{equation*}
	For an ideal $\mathfrak{a}$ of $\OO_S$, let $\rad(\mathfrak{a})$ denote the product of all prime ideals
	of $\OO_S$ which divide $\mathfrak{a}$. Then we may represent each element of $I(\OO_S)\otimes \dual{G}$ by an $l$-tuple 
	$(\mathfrak{a}_1, \ldots, \mathfrak{a}_l)$ of integral ideals of $\OO_S$, 
	such that $\mathfrak{a}_j \mid \rad(\mathfrak{a}_j)^{n_j-1}$ for all $1\leq j \leq l$.
	If $r\in \mathcal{R}$ represents the class of $(\mathfrak{a}_1,\ldots,\mathfrak{a_l})$ in $I(\OO_S)\otimes \dual{G}$, then
	from \eqref{eqn:gamma} we find that 
	\begin{align}	\label{eqn:rho}
	\sum_{x \in k^* \otimes \dual{G}} |g(x)|
	\ll \sum_{\substack{(\mathfrak{a}_1,\ldots,\mathfrak{a}_l)\\\mathfrak{a}_j\mid\rad{\mathfrak{a}_j^{M-1}}}}
	\frac{1}{N(\rad(\mathfrak{a}_1\cdots\mathfrak{a}_l))^{\sigma - \varepsilon}}
	= \sum_{\mathfrak{a}\unlhd\OO_S} \frac{\rho(\mathfrak{a})}{N(\mathfrak{a})^{\sigma - \varepsilon}},
	\end{align}
	where $M = \max\{n_j\}$, $N(\mathfrak{a})$ denotes the absolute norm of the ideal $\mathfrak{a}$, and 
	$$\rho(\mathfrak{a}) = |\mu(\mathfrak{a})|^2
	\#\{(\mathfrak{a}_1,\ldots,\mathfrak{a}_l)\ :\ \mathfrak{a}_j\unlhd \OO_S, \, \mathfrak{a}_j\mid\rad(\mathfrak{a}_j^{M-1}) \, \forall j, \, 	\rad(\mathfrak{a}_1\cdots\mathfrak{a}_l) = \mathfrak{a}\}.$$
By $\mathfrak{a}\unlhd\OO_S$, we mean that $\mathfrak{a}$ is an integral ideal of $\OO_S$.	
	It is easily seen that $\rho$ is multiplicative and that for any prime ideal $\mathfrak{p}$ we have
	$$\rho(\mathfrak{p}) = M^l - 1.$$
	Whence the sum in  \eqref{eqn:rho} converges, as required.
\end{proof}

We now come to the application of Poisson summation.

\begin{proposition} \label{prop:Poisson}
	Let $f=\prod_v f_v: \Hom(\Adele^* , G)\to \CC$ be a product of functions $f_v: \Hom(k_v^*, G) \to \CC$ such that for all but finitely many places $v$ of $k$,
	the function $f_v$ takes only the value $1$ on $\Hom(k_v^*/\OO_v^*,G)$. Assume that there exists some 
	$\gamma>0$ such that $|f_v(\chi_v)| \leq \gamma$ for all places $v$ and all $\chi_v \in \Hom(k_v^*,G)$.
	
	Then for $\re s > 1/\alpha(G)$ the Fourier
	transform $\widehat{f}(\cdot;s)$ of $f\Phi_G^{-s}$ exists and we have the Poisson formula
        \begin{equation}\label{eq:prop_poisson}
	 \sum_{\chi \in \Hom(\Adele^*/k^{*} , G)} \frac{f(\chi)}{\Phi_G(\chi)^{s}}
	=\frac{1}{|\OO_k^*\otimes\dual{G}|} 
	\sum_{x \in k^*\otimes \dual{G}}\widehat{f}(x;s), \quad \re s > 1/\alpha(G).
	     \end{equation}
 Moreover, the right-hand side defines a holomorphic function on $\re s > 1/\alpha(G)$.
 
	Suppose furthermore  that there exists a finite set of places $S$ such that 
	$f_v$ is $\Hom(k_v^*/\OO_v^* , G)$-invariant for all $v \notin S$. Then 
	$\widehat{f}(x;s)$ is zero if $x \notin \mathcal{U}_S(G)$
	(see \eqref{def:U_S(G)}). In particular, the above sum over $x$ is finite.
\end{proposition}
\begin{proof}
	\emph{Step 1: Topological considerations}.
	We let 
	$A=\Hom(\Adele^* , G)$, we let \mbox{$B = \Hom(\Adele^*/k^{*} , G)$},
	and set $C= A/B$. The topology on $B$ induced from $A$ coincides with the compact-open topology and is discrete. In particular, $B$ is a closed subgroup of $A$.
We equip $B$ with the counting measure $\mathrm{d}b$.
	Next, we have the following exact sequence of locally compact abelian groups, where each morphism is open onto its image
	\begin{equation} \label{eqn:GWW}
		1 \to \Sha(k, \Cdual{G}) \to k^*\otimes \dual{G}  \to \Adele^*\otimes \dual{G}
		\to \Adele^*/k^*\otimes \dual{G} \to 1.
	\end{equation}
Here $\Cdual{G}=\Hom(\dual{G}, \Gm)$ denotes the Cartier dual of $\dual{G}$. To check our claims, we may assume that $G=\mu_n$. In this case, the result follows from the canonical topological isomorphism $X \otimes \ZZ/n\ZZ \cong X/nX$, which holds for every locally compact abelian group $X$ for which $nX$ is a closed subgroup (see Section \ref{subsec:top}), together with Kummer theory, which identifies the kernel of $k^*/k^{*n}\to \Adele^*/\Adele^{*n}$ with  $\Sha(k, \mu_n)$. Next, it follows from \cite[Thm.~9.1.3]{NSW00}
	that $\Sha(k, \Cdual{G})$ is a finite elementary abelian $2$-group. In particular, 
        \begin{equation}\label{eq:cdual}
	\dual{C} \cong (k^*\otimes \dual{G})/\Sha(k,\Cdual{G})
      \end{equation}
 is discrete. We equip $\dual{C}$ with the dual of the quotient measure on $C$
	induced by $\mathrm{d}\chi$ and $\mathrm{d}b$; this is some constant
	$\omega$ times the counting measure on $\dual{C}$.	
	
	\emph{Step 2: Poisson summation}.
	We shall use the version of Poisson summation given in the Corollary of
	\cite[\S II.8, p.~127]{Bou67}.  Note that the existence of the Fourier 
	transform $\widehat{f}(\cdot;s)$ for $\re s > 1/\alpha(G)$ follows from
	Lemma \ref{lem:Euler_product}.
	In order to apply Poisson summation, it suffices to show that
	\begin{enumerate}
		\item The sum \label{cond:1}
		$$
		\sum_{x \in k^*\otimes \dual{G}}|\widehat{f}(x;s)|
		$$
		exists for $\re s > 1/\alpha(G)$.
		\item  For $\re s > 1/\alpha(G)$, the sum \label{cond:2}
		$$\sum_{\chi \in \Hom(\Adele^*/k^{*} , G)} \frac{|f(\chi a)|}{|\Phi_G(\chi a)^{s}|}$$
		exists for all $a \in \Hom(\Adele^* , G)$ and defines a continuous function on $\Hom(\Adele^* , G)$.

	\end{enumerate}
	In light of Lemma \ref{lem:zero}, the validity of Condition \eqref{cond:1} follows from
	Lemma~\ref{lem:k*_sum}, which also shows that the right-hand side in \eqref{eq:prop_poisson} defines a holomorphic function on $\re s > 1/\alpha(G)$. As for Condition \eqref{cond:2}, let $S$ be a finite set of places of $k$
	containing those places $v$ with $f_v(\Hom(k_v^*/\OO_v^* , G)) \neq \{1\}$ and large enough 
	such that $\Pic \OO_S = 0$.  Then $\Adele^*/k^* \cong \Adele^*_S / \OO_S^*$. We define local functions \mbox{$\rho_v : \Hom(k_v^*,G)\to \CC$} by
	$$\rho_v(\chi_v) = 
	\begin{cases}
		1 , & \chi_v \in \Hom(k_v^*/\OO_v^*,G),\\
		\gamma, & \chi_v \not\in \Hom(k_v^*/\OO_v^*,G),
	\end{cases}
	$$
	for $v \not \in S$ and $\rho_v(\chi_v) = \gamma$ for $v \in S$. Recall that, by continuity,
        \begin{equation*}
          \Hom(\Adele_S^*,G) = \bigoplus_{v\in S}\Hom(k_v^*,G)\oplus \bigoplus_{v\notin S}\Hom(\OO_v^*,G).
        \end{equation*}
Since $\rho_v$ is $\Hom( k_v^*/\OO_v^*, G)$-invariant for all $v$, this allows us to define the function 
	$\rho := \prod_v \rho_v: \Hom(\Adele^*_S , G)\to \CC$.
	We also use $\rho$ to denote the pullback of $\rho$ to $\Hom(\Adele^*,G)$, which dominates $f$. Since $\Hom(\Adele^*,G)$ is second-countable, it is enough to show continuity in terms of sequences. Let $(a_n)_n$ be a sequence  in $\Hom(\Adele^*,G)$ with limit $a$, and assume without loss of generality that 
\begin{equation*}
  a_n \in a\cdot \prod_{v\in S}\{\Const_{k_v^*}\}\prod_{v\notin S}\Hom(k_v^*/\OO_v^*,G)
\end{equation*}
holds for all $n\in \NN$. Then
\begin{equation*}
  \left|\frac{f(\chi a_n)}{\Phi_G(\chi a_n)^s} \right| \leq \frac{\rho(\chi a_n)}{\Phi_G(\chi a_n)^{\re s}} = \frac{\rho(\chi a)}{\Phi_G(\chi a)^{\re s}}
\end{equation*}
holds for all $\chi \in \Hom(\Adele^*/k^*,G)$. Moreover, 
	\begin{align*}
		& \sum_{\chi \in \Hom(\Adele^*/k^{*} , G)} \frac{\rho(\chi a)}{\Phi_G(\chi a)^{\re s}}
		= \sum_{\chi \in \Hom(\Adele^*_S/\OO_S^{*} , G)} \frac{\rho(\chi a)}{\Phi_G(\chi a)^{\re s}} \\		
		&\ll  \sum_{\chi \in \Hom(\Adele^*_S , G)} \frac{\rho(\chi a)}{\Phi_G(\chi a)^{\re s}} 
		= \sum_{\chi \in \Hom(\Adele^*_S , G)} \frac{\rho(\chi)}{\Phi_G(\chi)^{\re s}} \\
		& =  \prod_{v \in S} \sum_{\chi_v \in \Hom(k_v^* , G)} \frac{\rho_v(\chi_v)}{\Phi_G(\chi_v)^{\re s}} 
		\prod_{v \notin S} \sum_{\chi_v \in \Hom(\OO_v^* , G)} \frac{\rho_v(\chi_v)}{\Phi_G(\chi_v)^{\re s}}	\\
		& =  \prod_{v \in S} \sum_{\chi_v \in \Hom(k_v^* , G)} \frac{\rho_v(\chi_v)}{\Phi_G(\chi_v)^{\re s}} 
		\prod_{v \notin S} \widehat{\rho}_v(\Const_v; \re s).
	\end{align*}
The absolute convergence of this for $\re s > 1/\alpha(G)$  follows from the bounds given in Lemma \ref{lem:zero}. Continuity then follows from Lebesgue's dominated convergence theorem,  and hence Condition \eqref{cond:2} holds. Using \eqref{eq:cdual}, Poisson summation yields
	\begin{equation} \label{eqn:Poisson}
	\sum_{\chi \in \Hom(\Adele^*/k^{*} , G)} \frac{f(\chi)}{\Phi_G(\chi)^{s}}
	= \omega\sum_{x \in \dual{C}}\widehat{f}(x;s)  = \frac{\omega}{|\Sha(k, \Cdual{G})|}
	\sum_{x \in k^* \otimes \dual{G}}\widehat{f}(x;s).
	\end{equation}
	
	\emph{Step 3: Calculation of $\omega$}.
	To complete the proof of \eqref{eq:prop_poisson}, it suffices to show
	\begin{equation} \label{eqn:omega}
		\omega = \frac{|\Sha(k, \Cdual{G})|}{|\OO_k^*\otimes \dual{G}|}.
	\end{equation}
	We shall do this by applying the Poisson formula to a special choice of $f$.
	For each place $v$ of $k$, let $f_v$ be the
	indicator function of $\Hom(k_v^*/\OO_v^* , G)$.
	Then
	\begin{align*}
		\sum_{\chi \in \Hom(\Adele^*/k^{*} , G)} \frac{f(\chi)}{\Phi_G(\chi)^{s}}
		&= \# \{\chi \in \Hom(\Adele^*/k^{*} , G) : 
		\chi_v \in \Hom(k_v^*/\OO_v^* , G) \, \forall v \}\\
		&= \# \ker\left( \Hom(\Adele^*/k^* , G) \to \bigoplus_{v}\Hom(\OO_v^* , G)\right).
	\end{align*}
	Consider the exact sequence of locally compact abelian groups
	$$ \prod_{v} \OO_v^* \to \Adele^*/k^* \to \Pic \OO_k \to 1.$$
	From this we obtain the exact sequence
	$$ 1 \to \Hom(\Pic \OO_k, G) \to \Hom(\Adele^*/k^*,G) \to  \bigoplus_{v}\Hom(\OO_v^* , G),$$
	thus
	\begin{equation} \label{eqn:1}
	\sum_{\chi \in \Hom(\Adele^*/k^{*} , G)} \frac{f(\chi)}{\Phi_G(\chi)^{s}} = |\Hom(\Pic \OO_k, G)|.
	\end{equation}
	Next,  we have
	\begin{align*} \label{eqn:card}
	\sum_{x \in k^*\otimes \dual{G}}\widehat{f}(x;s) 
	&= \#\{x \in k^*\otimes \dual{G} : x_v \in \OO_v^*\otimes \dual{G} \, \forall v\} 
	\\ &= |\OO_k^*\otimes\dual{G}| \cdot |\Tor( \Pic \OO_k, \dual{G})|,
	\end{align*}
	where the second equality follows from Lemma \ref{lem:U}.
	Combine this with \eqref{eqn:Poisson} and \eqref{eqn:1} to obtain \eqref{eqn:omega}.
	The rest of the proposition follows from Lemma \ref{lem:zero} and Lemma \ref{lem:U}.
\end{proof}

\begin{remark} \label{rem:Wright}
	In Wright's setting (which corresponds to $f = \Const$), 
	he was in the second case of Proposition \ref{prop:Poisson}.
	He used character orthogonality to obtain
	a finite sum in \cite[\S3]{Wri89} (see in particular \cite[(3.2)]{Wri89}).
	For more general applications, however, 
	it may happen that the right hand sum in Proposition \ref{prop:Poisson} is infinite
	(this occurs for example in the proof of Theorem \ref{thm:conditions}).
	Notice also the consideration of the factor $\Sha(k, \Cdual{G})$ in the proof of Proposition~\ref{prop:Poisson},
	which was overlooked in Wright's work \cite[Lem.~3.1]{Wri89}.
	Lemma \ref{lem:U} may be viewed as a corrected version of \cite[Lem.~3.1]{Wri89}.
\end{remark}

\section{Proof of Theorem \ref{thm:conditions}} \label{sec:conditions}

\subsection{Set-up} \label{sec:set-up}
For any place $v$ of $k$ and any sub-$G$-extension $\varphi_v$ of $k_v$, let
\begin{equation*}
  f_v(\varphi_v) :=
  \begin{cases}
    1&\text{ if }\varphi_v \in \Lambda_v\\
    0&\text{ otherwise,}
  \end{cases}
\end{equation*}
where the set $\Lambda_v$ is as in Theorem \ref{thm:conditions}. For any sub-$G$-extension $\varphi$ of $k$, let $(\varphi_v)_v$ be the induced local sub-$G$-extensions, and define
\begin{equation*}
  f(\varphi):= \prod_{v}f_v(\varphi_v),
\end{equation*}
the product running over all places $v$ of $k$. 

Throughout this section, we use the characterization of a sub-$G$-extension $\varphi$ as a continuous homomorphism $\Adele^*/k^*\to G$, where $G$ is endowed with the discrete topology. The interplay between local and global class field theory shows that for each place $v$, the sub-$G$-extension induced by the natural embedding $k_v^*\to \Adele^*$ is $\varphi_v$. 

Theorem \ref{thm:conditions} will follow from a standard Tauberian theorem once we prove that the Dirichlet series
  $D_{G,f}(s)$ defined in \eqref{eq:discriminant_series} gives a meromorphic function for $\re s > 1/\alpha(G)-\gamma$, for some $\gamma>0$, and has its rightmost pole at $s = 1/\alpha(G)$ of order $\nu(k,G)$. The current section is devoted to the proof of this result, starting with applications of the general techniques from Sections \ref{sec:disc} and \ref{sec:harmonic}. From Lemma \ref{lem:discriminant_to_conductor}, we see that, formally,
\begin{equation}\label{eq:positive_density_delsarte}
  D_{G,f}(s) = \sum_{H\subseteq G}\mu(G/H)F_{H,f}(s|G|/|H|),
\end{equation}
with the conductor series
  $F_{H,f}(s)$
defined as in \eqref{eq:conductor_series}.

We will soon see that $F_{H,f}(s|G|/|H|)$ converges absolutely for $\re s > 1/\alpha(G)$ (which also follows directly from Wright's result). Since the sum over all subgroups $H$ of $G$ in \eqref{eq:positive_density_delsarte} is finite, a meromorphic continuation of the conductor series $F_{H,f}(s)$ will yield the required analytic properties of $D_{G,f}(s)$, except for the possible cancellation of poles. In Subsection \ref{subsec:positive_density_constant}, we will show that, under the assumptions of Theorem \ref{thm:conditions}, no such cancellation occurs.

\subsection{Poisson-summation}
Fix a subgroup $H$ of $G$. We now apply the results of Section \ref{sec:harmonic} to the conductor series $F_{H,f}(s)$.
The assumption \eqref{eqn:conditions} implies that for all but finitely many $v$, the $\Lambda_v$ contain all unramified sub-$G$-extensions of $k_v$. Hence the product  $f(\chi) := \prod_{v}f_v(\chi_v)$ gives a well-defined continuous function $\Hom(\Adele^*, H)\to \CC$. Since $f$ clearly satisfies the hypotheses of Proposition \ref{prop:Poisson} (with $G$ replaced by $H$), we obtain
\begin{equation}\label{eq:poisson_result}
  F_{H,f}(s) = \frac{1}{|\OO_k^*\otimes\dual{H}|}\sum_{x\in k^* \otimes \dual{H}}\widehat{f}_H(x;s),\quad \text{ for }\re s > 1/\alpha(H),
\end{equation}
where $\widehat{f}_H(x;s)$ denotes the Fourier transform of $f\Phi_H^{-s}$. In particular, the complex function $F_{H,f}(x;s|G|/|H|)$ is holomorphic on $\re s > |H|/(\alpha(H)|G|)$. 

Let $Q$ be the smallest prime dividing $|G|$. 
If $Q\nmid |H|$, then $\alpha(H)|G|/|H|\geq \alpha(G)+1$, so $F_{H,f}(x;s|G|/|H|)$ is holomorphic on $\re s>1/(\alpha(G)+1)$.
If $Q\mid |H|$, then we will show that $F_{H,f}(x;s|G|/|H|)$ has a meromorphic continuation  to the left of $\re s = 1/\alpha(G)$. 

The main difference between our situation and \cite[\S4, \S5]{Wri89}, is that in our case the sum over $x$ in \eqref{eq:poisson_result} may be infinite, whereas the analogue in Wright's situation is a finite sum (cf.~Remark \ref{rem:Wright}). Hence, in our analysis of $\widehat{f}_H(x;s)$ it is vitally important to control the dependence on $x$. In particular, in what follows all implied constants in Landau's $O$-notation and Vinogradov's $\ll$-notation will not depend on $x$ or $s$. Aside from this, our strategy follows the same classical steps of relating the resulting Euler products to Dedekind zeta functions.

\subsection{Euler factors} \label{subsec:Euler}
In this subsection, we will assume that $Q\mid |H|$. By Lemma \ref{lem:Euler_product}, we obtain the Euler product
$\widehat{f}_H(x;s) = \prod_v \widehat{f}_{H,v}(x_v;s)$, with
\begin{equation} \label{eq:positive_density_euler_factors}
   \widehat{f}_{H,v}(x_v;s) := \frac{1}{|H|}\sum_{\chi_v \in \Hom(k_v^*, H)}\frac{f_v(\chi_v)\pair{\chi_v}{x_v}}{\Phi_H(\chi_v)^{s}}
\end{equation}
for non-archimedean $v$.

We now evaluate the leading terms of the Euler factors $\widehat{f}_{H,v}(x_v;s)$ at almost all places of $k$. We allow ourselves to exclude the places in a finite set $S$ which contains all archimedean places, all places dividing the order of $G$, all places $v$ where $\Lambda_v$ does not satisfy \eqref{eqn:conditions}, and enough places to ensure that the ring of $S$-integers $\OO_S$ is a principal ideal domain. At later points, we will enlarge $S$ as necessary, but it will never depend on anything but $k$, $G$ and our collection $\Lambda$ of local conditions.

Let $\beta_H\geq 1$ be the largest integer such that $(\ZZ/Q\ZZ)^{\beta_H}$ is isomorphic to a subgroup of $H$.
We set $\beta := \beta_G$.
Let $x_v \in \OO_v^* \otimes \dual{H}$. We write $x_v\in \OO_v^{*Q}\otimes \dual{H}$ to say that $x_v$ is in the image of $\OO_v^{*Q}\otimes \dual{H}$ in $\OO_v^*\otimes \dual{H}$ under the natural map (which is not injective in general).


The following is inspired by \cite[Prop.~4.3]{Wri89}.
\begin{lemma}\label{lem:local_positive_density}
  Let $v\notin S$ and assume that $Q\mid |H|$. Let $\re s\geq 0$. 

If $x_v \in \OO_v^*\otimes \dual{H}$, then $\widehat{f}_{H,v}(x_v;s|G|/|H|)$ equals
  \begin{alignat*}{2}
      &1 + (Q^{\beta_H}-1)q_v^{-\alpha(G)s} + O(q_v^{-(\alpha(G)+1)s}),&&\text{ if } q_{v} \equiv 1 \bmod Q \text{ and }x_v\in \OO_v^{*Q}\otimes \dual{H},  \\
    &  1 - q_{v}^{-\alpha(G)s} + O(q_v^{-(\alpha(G)+1)s}),&&\text{ if }q_{v} \equiv 1\bmod Q\text{ and }x_v\not\in \OO_v^{*Q}\otimes \dual{H}, \\
  &    1 + O(q_v^{-(\alpha(G)+1)s}),&&\text{ if }q_{v}\not\equiv 1\bmod Q.
  \end{alignat*}
  If $x_v \not\in \OO_v^*\otimes \dual{H}$, then $\widehat{f}_{H,v}(x_v;s|G|/|H|) = O(q_v^{-(\alpha(G)+1)s})$. 
\end{lemma}

\begin{proof}
From Lemma \ref{lem:local}, with $G$ replaced by $H$, we get



\begin{equation*}
  \widehat{f}_{H,v}(x_v;s) = \sum_{\substack{m\mid d_{H,v}\\m\in\{1,Q\}}}
  \sum_{\substack{\chi_v\in\Hom(\OO_v^*,H)\\ \ker\chi_v=\OO_v^{*m}}}
  \hspace{-15pt}
  \pair{\chi_v}{x_v}\tau_{f_v}(\chi_v, x_v) q_v^{-|H|(1-1/m)s} + O(q_v^{-(\alpha(H)+1)s}),
\end{equation*}
where $d_{H,v}$ is the greatest common divisor of $q_v - 1$ and the exponent of $H$. By our conditions \eqref{eqn:conditions} on $\Lambda_v$, we have $f_v(\psi_v)=1$ for $\psi_v\in \Hom(k_v^*/\OO_v^*, H)$. As in the proof of Lemma \ref{lem:zero}, we see that $\tau_{f_v}(\Const_v,x_v) = 1$ if $x_v \in \OO_v^*\otimes \dual{H}$ and $\tau_{f_v}(\Const_v,x_v) = 0$ otherwise.

If $q_v\not\equiv 1\bmod Q$, then $Q\nmid d_{H,v}$, and hence we obtain
\begin{equation*}
  \widehat{f}_{H,v}(x_v; s) =
  \begin{cases}
    1+O(q_v^{-(\alpha(H)+1)s}), &\text{ if }x_v \in \OO_v^*\otimes \dual{H},\\
    O(q_v^{-(\alpha(H)+1)s}), &\text{ if }x_v \notin \OO_v^*\otimes \dual{H}.
  \end{cases}
\end{equation*}
Now assume that $q_v\equiv 1\bmod Q$, so that $Q\mid d_{H,v}$. Let $\chi_v\in\Hom(\OO_v^*, H)$ with $\ker\chi_v = \OO_v^{*Q}$. The conditions \eqref{eqn:conditions} on $\Lambda_v$ ensure that $f_v(\chi_v\psi_v)=1$ for all $\psi_v\in\Hom(k_v^*/\OO_v^*,H)$. Hence,
\begin{equation*}
  \tau_{f_v}(\chi_v, x_v) := \frac{1}{|H|}\sum_{\psi_v \in \Hom(k_v^*/\OO_v^*, H)} \pair{\psi_v}{x_v} =
    \begin{cases}
      1 &\text{ if } x_v \in \OO_v^*\otimes \dual{H},\\
      0 &\text{ otherwise.}
    \end{cases}
\end{equation*}
Thus, we have shown that
\begin{equation*}
  \widehat{f}_{H,v}(x_v; s) =
  \begin{cases}
    1 + \sigma_v(x_v)q_v^{-\alpha(H)s} + O(q_v^{-(\alpha(H)+1)s}) &\text{ if }x_v \in \OO_v^*\otimes \dual{H},\\
    O(q_v^{-(\alpha(H)+1)s}) &\text{ if }x_v \notin \OO_v^*\otimes \dual{H},
  \end{cases}
\end{equation*}
where
\begin{equation*}
  \sigma_v(x_v):= \sum_{\substack{\chi_v\in\Hom(\OO_v^*, H)\\ \ker\chi_v=\OO_v^{*Q}}}\pair{\chi_v}{x_v}.
\end{equation*}

It is a simple task to evaluate $\sigma_v(x_v)$. Indeed, using character orthogonality, $\sigma_v(x_v)+1$ equals
\begin{equation*}
  \sum_{\substack{\chi_v\in\Hom(\OO_v^*/\OO_v^{*Q}, H)}}\pair{\chi_v}{x_v} =
  \begin{cases}
    |\Hom(\OO_v^*/\OO_v^{*Q}, H)| &\text{ if }x_v\in \OO_v^{*Q}\otimes \dual{H} ,\\
    0 &\text{ if }x_v\not\in \OO_v^{*Q}\otimes \dual{H}.
  \end{cases}
\end{equation*}
Moreover, $|\Hom(\OO_v^*/\OO_v^{*Q}, H)| = Q^{\beta_H}$, the number of elements of order dividing $Q$ in $H$. We finish our proof with the observation that $\alpha(H)|G|/|H| = \alpha(G)$, so that $(\alpha(H)+1)|G|/|H|\geq\alpha(G)+1$.
\end{proof}

\subsection{Splitting conditions}\label{subsec:positive_density_splitting_conditions}
Assume that $Q\mid|H|$. Our goal is to relate the Fourier transform $\widehat{f}_H(x; s)$ to a product of Dedekind zeta functions. Hence, we want to express the conditions deciding the shape of the Euler factors in Lemma~\ref{lem:local_positive_density} in terms of splitting properties of $v$.
Recall that $x_v$ denotes the image of $x \in k^*\otimes \dual{H}$ under the natural map $k^*\otimes \dual{H} \to k_v^*\otimes \dual{H}$. Let
\begin{equation}\label{eq:positive_density_T}
  T := T(x) := S \cup \{v \text{ place of }k :\ x_v\notin \OO_v^*\otimes \dual{H}\}.
\end{equation}
Let $k_0 := k(\mu_Q)$ be the field obtained by adjoining to $k$ all $Q$-th roots of unity.

Each $\psi\in\Hom(\mu_Q, H)$ induces a map $\Psi \in \Hom(k^*\otimes \dual{H}, k_0^*/k_0^{*Q})$ as follows:
\begin{equation}\label{eq:xtok_x}
k^*\otimes \dual{H}\to k_0^*/k_0^{*Q}\otimes \dual{H} 
\xrightarrow{1 \otimes \psi^\wedge} k_0^*/k_0^{*Q} \otimes \dual{\mu_Q}= k_0^*/k_0^{*Q}.
\end{equation}
Thus, by Kummer theory, each $x\in k^* \otimes \dual{H}$ determines an elementary abelian $Q$-extension $k_x/k_0$, corresponding to the subgroup of $k_0^*/k_0^{*Q}$ generated by the images of $x$ under all such $\Psi$. Let us collect some properties of the extensions $k_0$ and $k_x$ (cf.~the end
of \cite[\S4]{Wri89}).

\begin{lemma}\label{lem:positive_density_splitting}
Let $x\in k^*\otimes \dual{H}$.
\begin{enumerate}
\item The extension $k_0/k$ is cyclic of degree dividing $Q-1$, and the extension $k_x/k_0$ is abelian with Galois group isomorphic to a subgroup of $(\ZZ/Q\ZZ)^{\beta_H}$.
\item The conductor of $k_x/k_0$ has absolute norm $\ll \prod_{v\in T(x)}q_v^{[k_0:k]}$.
\item For any place $v$ of $k$ with $v \notin T(x)$, we have the equivalences
\begin{align*}
  \text{ $v$ splits completely in $k_0$}& \quad \Leftrightarrow \quad q_v\equiv 1\bmod Q,\\
  \text{ $v$ splits completely in $k_x$}& \quad \Leftrightarrow \quad q_v\equiv 1 \bmod Q \text{ and }x_v\in \OO_v^{*Q} \otimes \dual{H}.
\end{align*}
\end{enumerate}
\end{lemma}

\begin{proof}
The $\mathbb{F}_Q$-vector subspace of $k_0^*/k_0^{*Q}$ corresponding to $k_x/k_0$ is spanned by the images of $x$ in $k_0^*/k_0^{*Q}$ under the $\beta_H$ maps coming from an $\mathbb{F}_Q$-basis of $\Hom(\mu_Q, H)$. The statement on the $\mathbb{F}_Q$-rank of $\Gal(k_x/k_0)$ then follows from Kummer theory. All other statements in (1) are obvious.

Now let $v$ be a non-archimedean place of $k$ and let $w$ be a place of $k_0$ above $v$.
If $v \notin T(x)$, then $x_v \in \OO_v^* \otimes \dual{H}$ and it follows that the induced extension
of $k_{0,w}$ is unramified. On the other hand, if $v \in T(x)$ then ramification may occur, but when $w \nmid Q$
any such ramification must be tame. This shows (2).

We now prove (3). It is well known that $v$ splits completely in $k_0$ if and only if $q_v\equiv 1\bmod Q$. Let  $v\notin T(x)$ be completely split in $k_0$, so that \mbox{$k_v=k_{0,w}$}. A variant of the construction \eqref{eq:xtok_x} associates to each $\psi\in \Hom(\mu_Q,H)$ a map \mbox{$\Psi_v \in \Hom(k_v^*\otimes \dual{H}, k_v^*/k_v^{*Q})$}, and  $v$ is completely split in $k_x/k$ if and only if \mbox{$x_v \in \ker \Psi_v$} for all $\psi$. Therefore, as $x_v \in \OO_v^* \otimes \dual{H}$, to prove the result it suffices to show that 
\begin{equation} \label{eqn:ker}
\bigcap_{\mathclap{\psi\in \Hom(\mu_Q,H)}} \,\, \ker \Psi_v = \ker\left( k_v^* \otimes \dual{H} \to k_v^*/k_v^{*Q} \otimes \dual{H}\right),\end{equation}
since $\ker( k_v^* \otimes \dual{H} \to k_v^*/k_v^{*Q} \otimes \dual{H}) 
= \im(k_v^{*Q} \otimes \dual{H} \to k_v^{*} \otimes \dual{H}).$
The equality \eqref{eqn:ker} follows from the fact that the map
$$k_v^*/k_v^{*Q} \otimes \dual{H} \longrightarrow \prod_{\mathclap{\psi\in \Hom(\mu_Q,H)}} k_v^*/k_v^{*Q},$$
induced by the product of the $\psi\in \Hom(\mu_Q,H)$, is injective.
\end{proof}

\subsection{$L$-functions}\label{subsec:positive_density_l_functions}
We continue to assume that $Q\mid |H|$. We now find a meromorphic continuation of the partial Euler product
\begin{equation*}
  L_{f,H}(x;s|G|/|H|):=\prod_{v\notin T(x)}\widehat{f}_{H,v}(x_v;s|G|/|H|)
\end{equation*}
to the left of $\re s = 1/\alpha(G)$, making the dependence on $x$ explicit. We will take care of the contribution of places
$v \in T(x)$ in the next subsection.

To continue, we introduce some notation.
Let $a,b>0$ and let $f(x;s)$, $g(x;s)$ be families of holomorphic functions on some half-plane $\re s > b \geq a$, parameterised by $x$. We write $f(x;s)\approx_a g(x;s)$ if there is a family of holomorphic functions $\phi(x;s)$ on the half-plane $\re s > a$, satisfying $1\ll_\epsilon\phi(x;s)\ll_\epsilon 1$ on $\re s > a+\epsilon$, independently of $x$, such that $f(x;s)=\phi(x;s)g(x;s)$ for $\re s > b$. 

For any finite Galois extension $K/k$, the partial Euler product 
\begin{equation*}
  \zeta_{K,T}(s) := \prod_{\substack{w\in\Omega_K\\w|_k\not\in T}}\frac{1}{1-q_w^{-s}},
\end{equation*}
of the Dedekind zeta function defines a meromorphic function with a simple pole at $s=1$. We use Lemma \ref{lem:local_positive_density} and Lemma \ref{lem:positive_density_splitting}  to relate $L_{f,H}(x;s|G|/|H|)$ to a product formed from $\zeta_{k_0,T}(s)$ and $\zeta_{k_x,T}(s)$. 

\begin{lemma}\label{lem:positive_density_d_th_power}
Let $d:=[k_0:k]$ and $Q^m := [k_x : k_0]$, with $m = m(x)\leq \beta_H$, and $T = T(x)$ as in \eqref{eq:positive_density_T}. Then
\begin{equation}\label{eq:positive_density_d_th_power}
    L_{f,H}(x;s|G|/|H|)^d \approx_{1/(\alpha(G)+1)}
    \zeta_{k_x,T}(\alpha(G)s)^{(Q^{(\beta_H-m)}-1)}\left(\frac{\zeta_{k_x,T}(\alpha(G)s)}{\zeta_{k_{0},T}(\alpha(G)s)}\right).
\end{equation}
\end{lemma}

\begin{proof}
  The argument is fairly standard and can be found in more detail in \mbox{\cite[\S 5]{Wri89}}. In each step we might need to enlarge our set $S$ by some small places, to ensure that all occurring Euler factors are bounded away from $0$. We can choose these places independently of $x$ (and $T$), since all explicit and implicit constants appearing in the Euler factors are bounded independently of $x$. By Lemma \ref{lem:positive_density_splitting} and the usual argument about split primes, we obtain 
\begin{align*}
  \zeta_{k_0,T}(s) &\approx_{1/2}\prod_{\substack{v\notin T\\q_v\equiv 1\bmod Q}}\left(1+dq_v^{-s}\right)\text{, and}\\
  \zeta_{k_x,T}(s) &\approx_{1/2}\prod_{\substack{v\notin T\\q_v\equiv 1\bmod Q\\
  x_v\in \OO_v^{*Q}\otimes \dual{H}}}\left(1+dQ^mq_v^{-s}\right).
\end{align*}
Hence, the expression on the right-hand side of \eqref{eq:positive_density_d_th_power} is 
\begin{equation*}
\approx_{1/(2\alpha(G))} \prod_{\substack{v\notin T\\q_v\equiv 1\bmod Q
\\x_v\in \OO_v^{*Q}\otimes \dual{H}}}\left(1+d(Q^{\beta_H}-1)q_v^{-\alpha(G)s}\right)\prod_{\substack{v\notin T\\q_v\equiv 1\bmod Q\\x_v\not\in \OO_v^{*Q}\otimes \dual{H}}}\left(1 - dq_v^{-\alpha(G)s}\right).
\end{equation*}
Together with Lemma \ref{lem:local_positive_density}, this yields the desired result.
\end{proof}

We want to isolate the pole at $s=1/\alpha(G)$ on the right-hand side of \eqref{eq:positive_density_d_th_power} independently of $x$. This is achieved in the next lemma. For any number field $K$, we let $\zeta_K(s)$ be its Dedekind zeta function. 

\begin{lemma}
  The complex functions $\zeta_{k_x,T}(s)/\zeta_{k_0,T}(s)$ and $\zeta_{k_x,T}(s)/\zeta_{k_0}(s)$ are entire. For any $\delta > 0$ there is a positive constant $\gamma(\delta)$ that depends only on $\delta$, $k$, and $G$, such that
  \begin{equation}\label{eq:positive_density_zeta_bounds}
    \frac{\zeta_{k_x,T}(s)}{\zeta_{k_0,T}(s)} \ll_\delta (1+|\im s|)^\delta \prod_{v\in T}q_v^\delta\quad\text{ and }\quad     \frac{\zeta_{k_x,T}(s)}{\zeta_{k_0}(s)} \ll_\delta (1+|\im s|)^\delta \prod_{v\in T}q_v^\delta
  \end{equation}
  hold for $\re s > 1-\gamma(\delta)$, where the implied constant depends only on $k, G$ and $\delta$.
\end{lemma}

\begin{proof}
Let $\Adele_{k_0}^*$ be the idele group of $k_0$, and let $U$ be the open subgroup corresponding to the abelian extension $k_x/k_0$. Then
  \begin{equation*}
    \zeta_{k_x}(s) = \,\, \prod_{\mathclap{\substack{\chi \in \dual{(\Adele_{k_0}^*/U)}}}} \,L(s,\chi),
  \end{equation*}
where $L(s,\chi)$ is the $L$-function associated to $\chi$. Then $L(s,\Const) = \zeta_{k_0}(s)$, and $L(s,\chi)$ is entire for $\chi \neq \Const$. This shows that $\zeta_{k_x}(s)/\zeta_{k_0}(s)$ is entire. 

The conductor of every character $\chi \in \dual{(\Adele_{k_0}^*/U)}$ divides the conductor of $k_x/k_0$, so by Lemma \ref{lem:positive_density_splitting} its absolute norm is $\ll \prod_{v\in T(x)}q_v^d$, where $d = [k_0:k]$. The standard convexity bound (see \cite[(5.20)]{IK04}) yields, for $\chi\neq \Const$ and $\re s\in [0,1]$,
\begin{equation*}
  L(s,\chi) \ll_\epsilon \left((3 + |s|)^d\prod_{v\in T}q_v^d\right)^{(1-\re s)/2+\epsilon} \ll_\epsilon (1+|\im(s)|)^{2\epsilon}\prod_{v\in T}q_v^{2\epsilon},
\end{equation*}
for $1 - \gamma'(\epsilon)<\re s \leq 1$, with a suitable $\gamma'(\epsilon)>0$. This estimate extends to the half-plane $\re s > 1-\gamma'(\epsilon)$, by the Phragmen-Lindel\"of principle and the bound $L(s,\chi)\ll_\epsilon 1$ for $\re s > 1+\epsilon$. Since $|\dual{(\Adele_{k_0}^*/U)}| \leq Q^{\beta_H} \ll 1$, this shows that for any $\delta > 0$ there is a positive constant $\gamma(\delta)$, such that
\begin{equation*}
  \frac{\zeta_{k_x}(s)}{\zeta_{k_0}(s)} \ll_\delta (1+|\im s|)^{\delta}\prod_{v\in T}q_v^{\delta},
\end{equation*}
whenever $\re s> 1-\gamma(\delta)$. Let
\begin{equation*}
  P(T;s) := \prod_{v\in T\smallsetminus{\Omega_\infty}}\prod_{\substack{w \in \Omega(k_0)\\w|v}}\left(1-q_w^{-s}\right)^{-1}\prod_{\substack{w' \in\Omega(k_x)\\w'|w}}\left(1-q_{w'}^{-s}\right).
\end{equation*}
Then 
\begin{equation} \label{eqn:P(T;s)}
  \frac{\zeta_{k_x,T}(s)}{\zeta_{k_0,T}(s)} = P(T;s)\cdot\frac{\zeta_{k_x}(s)}{\zeta_{k_0}(s)}.
\end{equation}
Since $(1-X)$ divides $(1-X^{f})$ for any positive integer $f$, the function $P(T;s)$ is entire, and $|P(T;s)|\leq 2^{dQ^\beta |T|}$ for $\re s> 0$. 
This proves the first part of \eqref{eq:positive_density_zeta_bounds}. For the second part, use
\begin{equation*}
  \tilde{P}(T;s) := \prod_{\mathclap{\substack{w' \in\Omega(k_x)\\w'|_k \in T \setminus \Omega_\infty}}}\left(1-q_{w'}^{-s}\right). \qedhere
\end{equation*}

\end{proof}

 It is proved in \cite[\S5]{Wri89} that the $d$-th root of $\zeta_{k_x,T}(s)/\zeta_{k_{0},T}(s)$ can be defined as a single-valued holomorphic function
 (this is because each zero and pole occurs with multiplicity divisible by $d$). Hence, we have proved the following lemma,
 which is the main result of this subsection.

\begin{lemma}\label{lem:positive_density_euler_prod_analyt_cont}
For any $x\in k^*\otimes \dual{H}$, let $T, m,d$ be as in Lemma \ref{lem:positive_density_d_th_power}. Then there is a function $\phi(x;s)$, holomorphic on $\re s > 1/(\alpha(G)+1)$, such that 
\begin{equation*}
  L_{f,H}(x;s|G|/|H|) = \phi(x;s)\cdot \zeta_{k_0}(\alpha(G)s)^{(Q^{\beta_H-m}-1)/d}.   
\end{equation*}
Moreover, for any $\delta>0$ there is a positive constant $\gamma(\delta)$, independent of $x$, such that for $\re(s)>1/\alpha(G)-\gamma(\delta)$, we have
\begin{equation}\label{eq:positive_density_phi_bound}
  \phi(x;s)\ll_\delta (1+|\im s|)^\delta\prod_{v\in T}q_v^\delta.
\end{equation}
\end{lemma}
The highest order of the potential pole at $s=1/\alpha(G)$ is achieved when $m=0$, that is, when $k_x = k_0$. 
\subsection{Global analysis}
We now apply our results from the previous subsections to analyse the sum on the right-hand side of \eqref{eq:poisson_result}, still under the assumption that $Q\mid |H|$. Recall that for $x\in k^*\otimes \dual{H}$, we defined $m(x)\in \{0, \ldots, \beta_H\}$ by $[k_x:k_0] = Q^{m(x)}$. We sort all $x$ by their respective values of $m(x)$ and apply Lemma \ref{lem:positive_density_euler_prod_analyt_cont} to obtain
\begin{equation} \label{eqn:F_zeta}
  F_{H,f}(s|G|/|H|) = \frac{1}{|\OO_k^*\otimes\dual{H}|}\sum_{m=0}^{\beta_H}\zeta_{k_0}(\alpha(G)s)^{(Q^{\beta_H-m}-1)/d}S(m;s),
\end{equation}
where
\begin{equation*}
  S(m;s) := \sum_{\substack{x\in k^*\otimes \dual{H}\\m(x)=m}}\phi(x;s)\prod_{v\in T(x)}\widehat{f}_{H,v}(x_v; s|G|/|H|).
\end{equation*}
Since the sum over $m$ is finite, it suffices to consider each $S(m;s)$ separately. 

\begin{lemma}\label{lem:positive_density_absolute_convergence}
  Let $m \in \{0, \ldots, \beta_H\}$. There are constants $\delta,\gamma > 0$, depending only on $k, G$, such that the series $S(m;s)$ converges absolutely for $\re s > 1/\alpha(G)-\gamma$ and satisfies $S(m;s)\ll (1+|\im s|)^\delta$ in this half-plane. Moreover, the convergence is uniform for $s$ in compact discs, so $S(m;s)$ defines a holomorphic function for $\re s > 1/\alpha(G)-\gamma$. 
\end{lemma}

\begin{proof}
  Since we only consider absolute convergence, we may ignore the condition $m(x) = m$. Fix $\delta > 0$ and let $\gamma(\delta)$ be as in Lemma \ref{lem:positive_density_euler_prod_analyt_cont}. Then
  \begin{equation*}
    \phi(x;s)\ll_\delta (1+|\im s|)^\delta\prod_{v\in T(x)}q_v^\delta, \quad \re s> 1/\alpha(G)-\gamma(\delta).
  \end{equation*}
An inspection of \eqref{eq:positive_density_euler_factors} shows that $\widehat{f}_{H,v}(x_v; s|G|/|H|)$ is bounded independently of $x_v$. For $v\in T(x)\smallsetminus S$, we have the bound 
\begin{equation*}
  |\widehat{f}_{H,v}(x_v;s|G|/|H|)| \leq C q_v^{-(\alpha(G)+1)\re s}, \quad \re s\geq 0,
\end{equation*}
from Lemma \ref{lem:local_positive_density}, with $C>0$ depending only on $k$ and $G$. Fix $\gamma \in (0,\gamma(\delta))$. On $\re s > 1/\alpha(G)-\gamma$, 
 we obtain the majorant
  \begin{equation*}
   \phi(x;s)\prod_{v\in T(x)}\widehat{f}_{H,v}(x_v;s|G|/|H|)  \ll_{\delta}(1+|\im s|)^\delta \prod_{v}g_v(x_v),
  \end{equation*}
where $g_v(x_v):=1$ if $x_v\in \OO_v^* \otimes \dual{H}$ and
\begin{equation*}
  g_v(x_v) := Cq_v^{-(\alpha(G)+1)(1/\alpha(G)-\gamma) + \delta}
\end{equation*}
otherwise. Choosing $\delta$ and $\gamma$ small enough gives $(\alpha(G)+1)(1/\alpha(G)-\gamma)  - \delta> 1$. The result then follows from Lemma \ref{lem:k*_sum}.
\end{proof} 

We now come to the main result of this subsection.
\begin{lemma}\label{lem:analytic_properties_of_FHF}
  If $Q \mid |H|$, then the function $F_{H,f}(s|G|/|H|)$ has a meromorphic continuation to $\re s > 1/\alpha(G)-\gamma$. It is holomorphic in this half-plane, with the possible exception of a pole at $s=1/\alpha(G)$ of order at most $\nu(k,H)=(Q^{\beta_H}-1)/d$. Moreover, there is $\delta>0$ such that
  \begin{equation} \label{eq:FHF_bound}
    F_{H,f}(s|G|/|H|)(1-1/(\alpha(G)s))^{\nu(k,H)} \ll (1+|\im s|)^\delta,\,\, \re s>1/\alpha(G)-\gamma.
  \end{equation}
\end{lemma}
\begin{proof}
	Lemma \ref{lem:positive_density_absolute_convergence} and \eqref{eqn:F_zeta} show that $F_{H,f}(s|G|/|H|)$ is a finite sum of meromorphic functions on $\re s> 1/\alpha(G)-\gamma$ whose poles come from finitely many powers of $\zeta_{k_0}(\alpha(G)s)$. The estimate \eqref{eq:FHF_bound} follows from the estimate for $S(m;s)$ in Lemma \ref{lem:positive_density_absolute_convergence} and standard convexity bounds for Dedekind zeta functions.
\end{proof}

If $Q\nmid |H|$ then we already know that $F_{H,f}(s|G|/|H|)$ is holomorphic for $\re s > 1/(\alpha(G)+1)$. Together with the expansion \eqref{eq:positive_density_delsarte}, this immediately implies that $D_{G,f}(s)$ has a meromorphic continuation to $\re s > 1/\alpha(G)-\gamma$, with the rightmost possible pole at $s=1/\alpha(G)$ of order at most $\nu(k,G)$.

As we will show in the next subsection, the order of the pole is indeed $\nu(k,G)$. Then Theorem \ref{thm:conditions} follows immediately from standard power-saving Tauberian theorems, such as \cite[Th\'eor\`eme A.1.]{CLT}. 

\subsection{Positivity of the constant}\label{subsec:positive_density_constant}
In this subsection, we let $\alpha := \alpha(G)$ and $\nu := \nu(k,G)$. We now show that $D_{G,f}(s)$ does indeed have a pole of order $\nu$ at $s=1/\alpha$. To do this, it suffices to prove that the limit
\begin{equation*}
  c_f(k,G) := \lim_{s\to 1/\alpha}\zeta_{k_0}(\alpha s)^{-\nu}D_{G,f}(s)
\end{equation*}
is positive. Our argument is loosely based on \cite[\S6]{Wri89}.
From \eqref{eq:positive_density_delsarte}, \eqref{eq:poisson_result} and Lemma \ref{lem:positive_density_absolute_convergence}, we obtain
\begin{equation} \label{eqn:constant}
  c_f(k,G) = \sum_{H\subseteq G}\frac{\mu(G/H)}{|\OO_k^*\otimes \dual{H}|}
  \sum_{x\in k^* \otimes \dual{H}}\prod_{v}\frac{\widehat{f}_{H,v}(x_v;|G|/(|H|\alpha))}{\zeta_{k_0,v}(1)^{\nu}},
\end{equation}
where $\zeta_{k_0,v}(s)$ is the product of the Euler factors of $\zeta_{k_0}(s)$ at all places of $k_0$ lying above $v$ if $v$ is non-archimedean, and  $\zeta_{k_0,v}(s)=1$ otherwise.
We would like to express $c_f(k,G)$ as a sum of non-negative terms, by applying Poisson summation and M\"obius inversion backwards. In order to apply Poisson summation, however, we need to approximate the innermost summand by simpler functions. Let
\begin{equation*}
  \widehat{E}_{H,v}(x_v) := \frac{\widehat{f}_{H,v}(x_v;|G|/(|H|\alpha))}{\zeta_{k_0,v}(1)^{\nu}},
\end{equation*}
and, for a finite set $T\supseteq S$ of places of $k$,
\begin{equation*}
  \widehat{E}_{H,T,v}(x_v):=
  \begin{cases}
    \widehat{E}_{H,v}(x_v) &\text{ if }v\in T,\\
    1          &\text{ if }v\notin T \text{ and }x_v\in \OO_v^*\otimes \dual{H},\\
    0          &\text{ if }v\notin T \text{ and }x_v\notin \OO_v^*\otimes \dual{H}.
  \end{cases}
\end{equation*}
Moreover, for $x\in k^*\otimes \dual{H}$, let
\begin{equation*}
  \widehat{E}_{H}(x) := \prod_{v}\widehat{E}_{H,v}(x_v)\quad\text{ and }\quad \widehat{E}_{H,T}(x) := \prod_{v}\widehat{E}_{H,T,v}(x_v).
\end{equation*}
We observe that $\widehat{E}_{H,T}(x) = 0$ if $x\notin \OO_T^*\otimes \dual{H}$, so
\begin{equation}\label{eq:positive_dinsity_sum_L_H_T}
  \sum_{x\in k^*\otimes \dual{H}}\widehat{E}_{H,T}(x)
\end{equation}
exists for any finite set $T$, as the sum may be taken over $\OO_T^*\otimes \dual{H}$, which is a finite set. Moreover, by definition, $\widehat{E}_{H,T}(x)$ converges pointwise to $\widehat{E}_{H}(x)$, as $T\to \Omega_K$.
We now show something stronger.

\begin{lemma}
  We have
  \begin{equation*}
    \lim_{T\to\Omega_K}\sum_{x\in k^*\otimes \dual{H}}\widehat{E}_{H,T}(x) 
    = \sum_{x\in k^*\otimes \dual{H}}\widehat{E}_{H}(x).
  \end{equation*}
\end{lemma}

\begin{proof}
Let $v\notin S$ and recall the bounds from Lemma \ref{lem:local_positive_density} if $Q\mid |H|$ and from Lemma \ref{lem:zero} (with $H$ instead of $G$) if $Q\nmid |H|$. If $x_v\notin \OO_v^*\otimes \dual{H}$, they show that $\widehat{E}_{H,v}(x_v) \ll q_v^{-1-1/\alpha}$, and consequently also
\begin{equation*}
  |\widehat{E}_{H,T,v}(x_v)| \leq C q_v^{-1-1/\alpha},
\end{equation*}
with $C>0$ depending only on $k$ and $G$. If $x_v\in\OO_v^*\otimes \dual{H}$, we compare them with the bounds
\begin{equation}\label{eq:constant_positive_zeta_bounds}
  \zeta_{k_0,v}(1)^{-\nu} =
  \begin{cases}
    1 - (Q^{\beta} - 1)q_v^{-1} + O(q_v^{-2}) &\text{ if } q_v\equiv 1\bmod Q,\\
    1 + O(q_v^{-2}) &\text{ if }q_v\not\equiv 1 \bmod Q,
  \end{cases}
\end{equation}
and conclude that
\begin{equation*}
  \widehat{E}_{H,T,v}(x_v) = 1 - c(x_v)q_v^{-1} +  O(q_v^{-1-1/\alpha}),
\end{equation*}
for some $c(x_v)\in \{0, \ldots, Q^\beta\}$.
All explicit and implied constants are bounded independently of $x$, hence we may assume that the set $S$ of excluded places was chosen large enough to ensure that $\widehat{E}_{H,T,v}(x_v) \geq 0$ for all $v\notin S$ and all $x_v\in \OO_v^*\otimes \dual{H}$. Putting everything together, we see that $\widehat{E}_{H,T}(x)$ is
\begin{equation*}
 \ll\hspace{-0.3cm} \prod_{\substack{v\notin S\\x_v\notin \OO_v^*\otimes \dual{H}}}Cq_v^{-1-1/\alpha}\hspace{-0.3cm}\prod_{\substack{v\notin S\\x_v\in \OO_v^*\otimes \dual{H}}}(1 + O(q_v^{-1-1/\alpha})) \ll\hspace{-0.3cm} \prod_{\substack{x_v\notin \OO_v^*\otimes \dual{H}}}Cq_v^{-1-1/\alpha}.
\end{equation*}
The sum of this last bound over $x\in k^*\otimes \dual{H}$ converges by Lemma \ref{lem:k*_sum}, so the desired result follows from Lebesgue's dominated convergence theorem.
\end{proof}

By the last lemma, we get
\begin{equation*}
  c_f(k,G) = \lim_{T\to \Omega_k}c_{f,T}(k,G),
\end{equation*}
where
\begin{equation}\label{eq:positive_density_def_c_f_T}
  c_{f,T}(k,G) := \sum_{H\subseteq G}\frac{\mu(G/H)}{|\OO_k^*\otimes\dual{H}|}
  \sum_{x\in k^*\otimes \dual{H}}\widehat{E}_{H,T}(x).
\end{equation}
Observe that $\widehat{E}_{H,T}(x)$ defines a function on $\Adele^* \otimes \dual{H}$ which is zero outside
\begin{equation*}
	\Adele_T^*\otimes \dual{H} = \prod_{v\in T}(k_v^*\otimes \dual{H})\prod_{v\notin T}(\OO_v^*\otimes \dual{H}).
\end{equation*}
We now describe the inverse Fourier transform of $\widehat{E}_{H,T}(x)$. 
For $\chi = (\chi_v)_{v} \in \Hom(\Adele^*, H)$, let
\begin{equation*}
  E_{T}(\chi) := \prod_{v}E_{T,v}(\chi_v),
\end{equation*}
where
\begin{equation*}
  E_{T,v}(\chi_v) :=
  \begin{cases}
    \zeta_{k_0,v}(1)^{-\nu}f_v(\chi_v) &v \in T,\\
    1  &v\notin T \text{ and }\chi_v\in \Hom(k_v^*/\OO_v^*,H)\\
    0  &v\notin T \text{ and }\chi_v\notin \Hom(k_v^*/\OO_v^*,H).
  \end{cases}
\end{equation*}
By a simple local calculation, $\widehat{E}_{H,T}(x) = \widehat{E}_{T}(x;1/(|H|(1-Q^{-1})))$, where $\widehat{E}_{T}(x;s)$ is defined as at the start of Subsection \ref{subsec:fourier_transforms}, with $E_{T}$ for $f$ and $H$ for $G$. As in Proposition \ref{prop:Poisson}, we obtain
\begin{equation}\label{eq:positive_density_poisson_backwards}
  \frac{1}{|\OO_k^*\otimes\dual{H}|}\sum_{x\in k^*\otimes \dual{H}}\widehat{E}_{H,T}(x) = \sum_{\chi\in\Hom(\Adele^*/k^*,H)}\frac{E_{T}(\chi)}{\Phi_H(\chi)^{1/(|H|(1-Q^{-1}))}}.
\end{equation}
Indeed, if $Q\nmid |H|$ then we can apply Proposition \ref{prop:Poisson} directly. If $Q\mid |H|$ then we already know that the Fourier transform $\widehat{E}_{T}(x;s)$ exists for $s=1/\alpha(H)$. Moreover, the sums on both sides are finite, as the summands $\widehat{E}_{H,T}(x)$ and $E_T(\chi)$ are supported on compact sets, and both $k^*\otimes \dual{H}$ and $\Hom(\Adele^*/k^*,H)$ are discrete. Hence, conditions \eqref{cond:1} and \eqref{cond:2} in the proof of Proposition \ref{prop:Poisson} hold trivially.

Using Lemma \ref{lem:conductor_to_discriminant} to express the right-hand side of \eqref{eq:positive_density_poisson_backwards} in terms of discriminant series, and substituting this into \eqref{eq:positive_density_def_c_f_T}, we see that $c_{f,T}(k,G)$ equals
\begin{equation}\label{eq:positive_density_non_negative_sum}
  \sum_{H\subseteq G}\mu(G/H)\sum_{J\subseteq H}\sum_{\chi \in \text{J}\text{-ext}(k)}\frac{E_{T}(\chi)}{\Delta(\chi)^{1/(|J|(1-Q^{-1}))}} = \sum_{\chi\in\gextk}\frac{E_{T}(\chi)}{\Phi_G(\chi)^{1/\alpha}},
\end{equation}
where the last equality holds by Delsarte's inversion formula \cite{Delsarte}. Note that the right-hand side is a sum with non-negative terms. 

Recall our assumption in the statement of Theorem \ref{thm:conditions} that there exists a $G$-extension $\varphi$ of $k$ such that $f(\varphi)=1$.
Let $T_0$ be a finite set of places of $k$, containing $S$ and all places at which $\varphi$ is ramified, that is, all places where $\OO_v^*\not\subseteq \ker \varphi_v$. Furthermore, we choose $T_0$ big enough to ensure that
\begin{equation*}
  \Adele^*(T_0)/(\ker\varphi\cap\Adele^*(T_0))  \cong (\prod_{v\in T_0}k_v^*)/(\ker\varphi\cap \prod_{v\in T_0}k_v^*) \cong G.
\end{equation*} 
For any $T\supseteq T_0$, the $G$-extension $\varphi$ appears in the sum on the right-hand side of \eqref{eq:positive_density_non_negative_sum}, and $E_{T}(\varphi) > 0$. 


We now use $\varphi$ to construct many more $G$-extensions $\chi$ of $k$ that contribute to the sum on the right-hand side of \eqref{eq:positive_density_non_negative_sum}. Let $\chi \in \Hom(\Adele^*/k^*,G)$ be any sub-$G$-extension with the property that $\chi_v = \psi_{v}$ for all $v\in T_0$. Then we have $\ker \chi \cap \prod_{v\in T_0}k_v^* = \ker \varphi\cap\prod_{v\in T_0}k_v^*$, so $\Adele^*/\ker\chi \cong G$, and thus $\chi\in\gextk$. Therefore,
\begin{equation*}
c_{f,T}(k,G) \geq \sum_{\substack{\chi\in\Hom(\Adele^*/k^*,G)\\ \chi_v = \varphi_{v} \text{ for } v\in T_0}}\frac{E_{T}(\chi)}{\Phi_G(\chi)^{1/\alpha}}.
\end{equation*}
We construct $\chi \in \Hom(\Adele^*/k^*,G)$ with $\chi_v = \varphi_{v}$ for $v \in T_0$ and $E_{T}(\chi)>0$ as follows. For $\psi = (\psi_v)_v \in \Hom(\Adele^*, G)$, let
\begin{equation*}
  h_{v}(\psi_v) :=
  \begin{cases}
        1 &\text{ if } [\OO_v^*:(\ker \psi_v\cap \OO_v^*)] \in \{1,Q\},\\
        0 &\text{ if } [\OO_v^*:(\ker\psi_v\cap \OO_v^*)] \notin \{1,Q\},\\
  \end{cases}
\end{equation*}
and
\begin{equation*}
  h_{T,v}(\psi_v) :=
  \begin{cases}
    1 &\text{ if }\psi_v = \Const \text{ and }v\in T_0,\\
    0 &\text{ if }\psi_v \neq \Const \text{ and }v\in T_0,\\
    \zeta_{k_0,v}(1)^{-\nu}h_{v}(\psi_v) &\text{ if }v \in T\smallsetminus T_0,\\
    1 &\text{ if } \OO_v^*\subseteq \ker\psi_v \text{ and }v\notin T,\\
    0 &\text{ if } \OO_v^*\not\subseteq \ker\psi_v \text{ and }v\notin T.
  \end{cases}
\end{equation*}
Let $h_T := \prod_v h_{T,v}$. For any $\psi$ with $h_T(\psi) \neq 0$, let $\chi := \psi\varphi$. Then $\chi_v = \varphi_{v}$ for $v \in T_0$. Our condition \eqref{eqn:conditions} on $\Lambda_v$ and the fact that $\varphi$ is unramified outside $T_0$ show that $h_{T,v}(\psi_v) \neq 0$ implies $f_v(\chi_v) = 1$ for $v \in T$. Also, $\Phi_G(\chi_v) = \Phi_G(\varphi_{v})$ if $v\in T_0$, and $\Phi_G(\chi_v) = \Phi_G(\psi_v)$ if $v \notin T_0$. We have shown that
\begin{equation*}
  c_{f,T}(k,G) \geq \prod_{v\in T_0}\left(\zeta_{k_0,v}(1)^{-\nu}\Phi_G(\varphi_v)^{-1/\alpha}\right) \sum_{\substack{\psi\in\Hom(\Adele^*/k^*,G)}}\frac{h_{T}(\psi)}{\Phi_G(\psi)^{1/\alpha}}.
\end{equation*}
The factor in front of the sum is positive and does not depend on $T$. Hence, it is enough to show that
\begin{equation*}
  R := \lim_{T\to\Omega_k}\sum_{\substack{\psi\in\Hom(\Adele^*/k^*,G)}}\frac{h_{T}(\psi)}{\Phi_G(\psi)^{1/\alpha}}>0.
\end{equation*}
We observe that $h_{T}$ is supported on $\Hom(\Adele^*,G)(T)$, and that $h_{T,v}(\psi_v) = 1$ whenever $v\notin T$ and $\psi_v \in \Hom(k_v^*/\OO_v^*,G)$. In particular, the above sum is finite for every $T$.
As in Subsection \ref{subsec:fourier_transforms}, but with $s=1/\alpha$ and $h_{T,v}$ instead of $f_v$, we define the local and global Fourier transforms $\widehat{h}_{T,v}(x_v;1/\alpha)$ and  $\widehat{h}_T(x;1/\alpha)$ of $h_T\Phi_G^{-1/\alpha}$. Then
$\widehat{h}_T(x;1/\alpha) = \prod_v\widehat{h}_{T,v}(x_v;1/\alpha)$ and
\begin{equation*}
  \widehat{h}_{T,v}(x_v;1/\alpha) =
  \begin{cases}
    |G|^{-1}
 &\text{ if }v\in T_0,\\ 
    \zeta_{k_0,v}(1)^{-\nu}\widehat{h}_{v}(x_v;1/\alpha) &\text{ if }v\in T\smallsetminus T_0,\\
    1 &\text{ if } x_v\in\OO_v^*\otimes \dual{G} \text{ and }v\notin T,\\
    0 &\text{ if } x_v\notin \OO_v^*\otimes \dual{G} \text{ and }v\notin T,
  \end{cases}
\end{equation*}
where
\begin{equation*}
  \widehat{h}_{v}(x_v;1/\alpha) := \frac{1}{|G|}\sum_{\psi_v\in \Hom(k_v^*,G)}\frac{h_{v}(\psi_v)\pair{\psi_v}{x_v}}{\Phi_G(\psi_v)^{1/\alpha}}.
\end{equation*}
By the Poisson summation formula, as in Proposition \ref{prop:Poisson}, we obtain
\begin{equation}\label{eq:positive_density_constant_after_poisson}
   R = \lim_{T\to \Omega_k}\frac{1}{|\OO_k^*\otimes \dual{G}|}\sum_{x\in k^* \otimes \dual{G}}\widehat{h}_T(x;1/\alpha).
\end{equation}
The sum over $x$ is finite, since all summands for $x\notin \OO_T^*\otimes \dual{G}$ vanish. In fact, even more holds. Let us compute the local factors $\widehat{h}_v(x_v;1/\alpha)$. \begin{lemma}\label{lem:local_positive_density_2}
  Let $v\notin T_0$. If $x_v \in \OO_v^* \otimes \dual{G}$, then
  \begin{equation*}
    \widehat{h}_{v}(x_v;1/\alpha) = 
    \begin{cases}
      1 + (Q^{\beta}-1)q_v^{-1} &\text{ if } q_{v} \equiv 1 \bmod Q \text{ and }x_v\in \OO_v^{*Q}\otimes \dual{G},\\
      1 - q_{v}^{-1} &\text{ if }q_{v} \equiv 1\bmod Q\text{ and }x_v\notin\OO_v^{*Q}\otimes \dual{G},\\
      1 &\text{ if }q_{v}\not\equiv 1\bmod Q.
    \end{cases}
  \end{equation*}
  If $x_v \not\in\OO_v^*\otimes \dual{G}$, then $\widehat{h}_{v}(x_v;1/\alpha) = 0$.
  
\end{lemma}

\begin{proof}
  This follows from Lemma \ref{lem:local}, similarly to the proof of Lemma \ref{lem:local_positive_density}.
\end{proof}
Hence, $\widehat{h}_T(x;1/\alpha) = 0$ unless $x \in \OO_{T_0}^*\otimes \dual{G}$. Since this is independent of $T$, we may exchange limit and sum in \eqref{eq:positive_density_constant_after_poisson} to obtain
\begin{equation}\label{eq:positive_density_constant_final}
  R = \frac{1}{|\OO_k^*\otimes \dual{G}|\ |G|^{|T_0|}}\sum_{x\in \OO_{T_0}^*\otimes \dual{G}}\prod_{v\notin T_0}\zeta_{k_0,v}(1)^{-\nu}\widehat{h}_v(x_v;1/\alpha).
\end{equation}
Compare Lemma \ref{lem:local_positive_density_2} to the bounds \eqref{eq:constant_positive_zeta_bounds} for $\zeta_{k_0,v}(1)^{-\nu}$. For $x \neq \Const$, the Euler product in \eqref{eq:positive_density_constant_final} may diverge to $0$ or converge. However, it is clear that the Euler product converges for $x=\Const$. Possibly enlarging $T_0$ by some small places, we may assume that all occurring Euler factors are positive. Thus, the product for $x = \Const$ is positive and hence $R > 0$. This completes the proof of \mbox{Theorem \ref{thm:conditions}}. \qed


\section{Some applications of  Theorem \ref{thm:conditions}} \label{sec:bicyclic}
We now present two applications of Theorem \ref{thm:conditions}, upon which the proofs of Theorem \ref{thm:HNP_0} and Theorem \ref{thm:WA}
will be based. Given a finite abelian extension $K/k$ and a place $v$ of $k$, we denote by $D_v$ the decomposition group at $v$.

Even though Theorem \ref{thm:conditions} gives a positive proportion result, it can also be used
to prove $0\%$ results.

\begin{theorem} \label{thm:bicyclic}
	Let $G$ be a finite abelian group and let $Q$ be the smallest prime dividing $|G|$. Assume that $G$
	contains a subgroup $A$ which is isomorphic to $(\ZZ/Q\ZZ)^2$. Let $S$ be a finite
	set of  places of $k$. Then $0\%$ of $G$-extensions $(K/k,\psi)$ have the property that $\psi(D_v) \neq A$ for all $v \not \in S$.
\end{theorem}

Note that Theorem~\ref{thm:bicyclic} does not contradict Theorem \ref{thm:conditions}; one may check using Lemma \ref{lem:special_form} below that the crucial assumption \eqref{eqn:conditions} does not hold in this case.
The following will be used to prove Theorem \ref{thm:WA}.

\begin{theorem} \label{thm:cyclic}
	Let $G$ be a finite abelian group and let $Q$ be the smallest prime dividing $|G|$.
	Let $S$ be a finite	set of  places of $k$.
	Then as $\varphi$ varies over all $G$-extensions of $k$, the following hold.
	\begin{enumerate}
		\item If the $Q$-Sylow subgroup of $G$ is cyclic, then a positive proportion of
		the extensions $K_\varphi/k$ have the property that all of their decomposition groups are cyclic.
		\item If the $Q$-Sylow subgroup of $G$ is not cyclic, then $0\%$ of 
		the extensions $K_\varphi/k$ have the property that $D_v$ is cyclic for all $v \notin S$.
	\end{enumerate}
\end{theorem}

\subsection{Proof of Theorem \ref{thm:bicyclic}}
We are free to increase the size of $S$
so that it contains all archimedean places, all places dividing the order of $G$, 
and so that various Euler factors appearing in the following proofs are positive.
Let $\alpha:=\alpha(G)$ and let $\nu:=\nu(k,G)$.
We keep the notation $\beta$ and $\beta_H$, for a subgroup $H \subseteq G$, from Subsection \ref{subsec:Euler}.

Now let $R$ be a finite set of places of $k$ which contains $S$.
Let $$\Lambda_v := \{ \mbox{sub-$G$-extensions } \varphi_v \mbox{ of } k_v: \varphi_v(\Gal(\bar{k}_v/k_v)) \neq A\}, \quad v \in R \setminus S,$$
and let $\Lambda_v$ be the set of all sub-$G$-extensions of $k_v$ for $v \not \in R \setminus S$. Let $\Lambda_R = (\Lambda_v)_v$. We now apply Theorem \ref{thm:conditions}
to obtain
$$\lim_{B \to \infty}\frac{N(k, G, \Lambda_R, B)}{N(k, G,B)} =: c_R$$
for some constant $c_R \geq 0$. To prove Theorem \ref{thm:bicyclic}, it suffices to show that
\begin{equation} \label{eqn:lim}
	\lim_{R \to \Omega_k} c_R = 0.
\end{equation}
For each place $v$, let $f_{R,v}$ be the indicator function associated to $\Lambda_v$, as defined in Section \ref{sec:set-up}.
Then from \eqref{eqn:constant} we obtain
\begin{equation} 
\label{eqn:upper_bound}
  c_R  \ll \sum_{H\subseteq G}\sum_{x\in k^*\otimes \dual{H}}\prod_{v}\frac{|\widehat{f}_{R,H,v}(x_v;|G|/(|H|\alpha))|}{\zeta_{k_0,v}(1)^{\nu}}.
\end{equation}
We begin by calculating the local Fourier transforms at places in $R$.

\begin{lemma} \label{lem:local_bicyclic}
	Let $H \subseteq G$ be a subgroup for which $\beta_H = \beta$.
	Let $v\in R \setminus S$, let $\re s\geq 0$ and let $\delta_H := (Q^2 -1)(Q^2 - Q)/|H|$.
	If $x_v \in \OO_v^*\otimes \dual{H}$, then there exists some $c_v(x_v) \in \CC$ with 
	$|c_v(x_v)| \leq (Q^{\beta}-1 - \delta_H)$, such that 
  \begin{equation*}
    \widehat{f}_{R,H,v}(x_v;s|G|/|H|) = \begin{cases}
      1 + c_v(x_v)q_v^{-\alpha(G)s} + O(q_v^{-(\alpha(G)+1)s}),&\text{ if } q_{v} \equiv 1 \bmod Q,\\
      1 + O(q_v^{-(\alpha(G)+1)s}),&\text{ if }q_{v}\not\equiv 1\bmod Q.
    \end{cases}
  \end{equation*}
  If $x_v \not\in\OO_v^*\otimes \dual{H}$, then $\widehat{f}_{R,H,v}(x_v;s|G|/|H|) = O(q_v^{-(\alpha(G)+1)s})$. 
 \end{lemma}
\begin{proof}
	The proof is very similar to the proof of Lemma \ref{lem:local_positive_density},
	so we shall be brief. If $x_v \notin\OO_v^*\otimes \dual{H}$, then since we are not removing
	any unramified sub-$G$-extensions, a similar proof to Lemma \ref{lem:local_positive_density} gives the upper
	bound. So assume that $x_v \in\OO_v^*\otimes \dual{H}$.
	The case $q_{v}\not\equiv 1\bmod Q$ is handled as in  Lemma \ref{lem:local_positive_density}.
	From now on, we assume that $q_{v}\equiv 1\bmod Q$. We apply
  	Lemma \ref{lem:local} to obtain
    \begin{equation*}
  \widehat{f}_{R,H,v}(x_v;s|G|/|H|)) = 1 + \frac{1}{|H|} \sum_{\substack{\chi_v\in\Hom(k_v^*,H)\\\ker\chi_v \cap \OO_v^{*} =\OO_v^{*Q}}}
  \hspace{-10pt}
  f_{v}(\chi_v)\pair{\chi_v}{x_v}q_v^{-\alpha(G)s}+ O(q_v^{-(\alpha(G)+1)s}).
\end{equation*}
	Using $|\pair{\chi_v}{x_v}| \leq 1$ and the definition of $f_v$, we have 
	$$ \Big|\sum_{\substack{\chi_v\in\Hom(k_v^*, H)\\\ker\chi_v \cap \OO_v^{*} =\OO_v^{*Q}}}
  f_{v}(\chi_v)\pair{\chi_v}{x_v}\Big| \leq \#\left\{ \chi_v \in \Hom(k_v^*, H) : 
	\begin{array}{ll}
		\ker\chi_v \cap \OO_v^{*} =\OO_v^{*Q},\\
		\chi_v(k_v^*) \neq A.
  	\end{array}  	
  	\right\}$$
	However, as in the proof of Lemma \ref{lem:local_positive_density}, we have 
	$$\#\{ \chi_v \in \Hom(k_v^*, H) : \ker\chi_v \cap \OO_v^{*} =\OO_v^{*Q}\} = (Q^\beta -1)|H|.$$
	Moreover, as the ramification is tame, the condition $\chi_v(k_v^*) = A$ implies that $
	\ker\chi_v \cap \OO_v^{*} =\OO_v^{*Q}$. Hence it suffices to observe that
	$$\#\{ \chi_v \in \Hom(k_v^*, H) : \chi_v(k_v^*) = A\} = (Q^2 -1)(Q^2 - Q),$$
	as this equals the number of surjective homomorphisms $k_v^* \to A$. This proves the lemma.
\end{proof}

We use this to show the following.

\begin{lemma}\label{lem:bicyclic_dominated_convergence_conditions}
  Let $H$ be a subgroup of $G$, let $x \in k^*\otimes \dual{H}$ and let
  \begin{equation*}
    F(R,x) := \prod_{v}\frac{|\widehat{f}_{R,H,v}(x_v;|G|/(|H|\alpha))|}{\zeta_{k_0,v}(1)^{\nu}}.
  \end{equation*}
  Let
$$  T(x) = S \cup \{v \text{ place of }k :\ x_v\notin \OO_v^*\otimes \dual{H}\}.$$
   Then the following statements hold.
  \begin{enumerate}
  \item Pointwise in $x$, we have $\lim_{R \to \Omega_k} F(R,x) = 0$.
  \item  There is $C>0$ such that $F(R,x) \ll \prod_{v\in T(x)}Cq_v^{-1-1/\alpha}$, uniformly in $R$ and $x$.
  \end{enumerate}
\end{lemma}

\begin{proof}
  If $\beta_H\neq \beta$, then Lemma  \ref{lem:positive_density_euler_prod_analyt_cont} shows that  $F(R,x)=0$ for all $R,x$. So let us assume that $\beta_H = \beta$. We fix $R,x$ and define, for $v\in\Omega_k$,
  \begin{equation*}
    F_v := \frac{|\widehat{f}_{R,H,v}(x_v;|G|/(|H|\alpha))|}{\zeta_{k_0,v}(1)^{\nu}}.
  \end{equation*}

Then
\begin{equation*}
  F(R,x) = \prod_{v \in T(x)}F_v 
  \prod_{\substack{v\in\Omega_k\smallsetminus T(x)\\q_v\not\equiv 1\bmod Q}} F_v 
  \prod_{\substack{v\in R\smallsetminus T(x)\\q_v\equiv 1\bmod Q}}F_v 
  \prod_{\substack{v\in\Omega_k\smallsetminus(R\cup T(x))\\q_v\equiv 1\bmod Q}}F_v.
\end{equation*}
We now bound these four products explicitly in $R,x$, by comparing the bounds for $\widehat{f}_{R,H,v}$ from Lemma \ref{lem:local_positive_density} and Lemma \ref{lem:local_bicyclic} with the bounds for $\zeta_{k_0,v}(1)^{-\nu}$ from \eqref{eq:constant_positive_zeta_bounds}. We obtain
\begin{align}
  \prod_{\mathclap{v \in T(x)}}F_v &\ll \prod_{v\in T(x)}Cq_v^{-1-1/\alpha},\label{eq:bicyclic_1}\\
  \prod_{\mathclap{\substack{v\in R\smallsetminus T(x)\\q_v\equiv 1\bmod Q}}}F_v &
  \ll \prod_{\substack{v\in R\smallsetminus T(x)\\q_v\equiv 1\bmod Q}}\left(1 - \delta_H q_v^{-1} + O(q_v^{-1-1/\alpha})\right),\label{eq:bicyclic_2}\\
  \prod_{\mathclap{\substack{v\in\Omega_k\smallsetminus T(x)\\q_v\not\equiv 1\bmod Q}}} F_v &\ll 1,
  \qquad  \qquad
  \prod_{\mathclap{\substack{v\in\Omega_k\smallsetminus(R\cup T(x))\\q_v\equiv 1\bmod Q}}}F_v \ll 1,\nonumber  
\end{align}
where $\delta_H$ is as in Lemma \ref{lem:local_bicyclic}. Assertion (1) in the lemma follows from \eqref{eq:bicyclic_2} and Chebotarev's density theorem. 
Assertion (2) follows from \eqref{eq:bicyclic_1} and the fact that all other products are $\ll 1$.
\end{proof}

Now \eqref{eqn:lim}, and thus Theorem \ref{thm:bicyclic}, follows immediately from Lemma \ref{lem:k*_sum}, Lemma \ref{lem:bicyclic_dominated_convergence_conditions} and Lebesgue's dominated convergence theorem. 
\qed

\subsection{Proof of Theorem \ref{thm:cyclic}}

\noindent We require the following proposition.

\begin{proposition} \label{prop:solvable}
	Let $k$ be a number field and $G$ a finite solvable group. Then there exists a $G$-extension
	$(K/k,\psi)$ all of whose decomposition groups are cyclic.
\end{proposition}
\begin{proof}
	As explained in the proof of \cite[Thm.~2]{Son08}, Shafarevich constructed such an extension
	in his resolution of the inverse Galois problem for solvable groups \cite{NSW00}.
\end{proof}

To prove Theorem \ref{thm:cyclic}, first suppose that the $Q$-Sylow subgroup of $G$ is cyclic. For each place $v$,
let $\Lambda_v$ be the set of all cyclic sub-$G$-extensions of $k_v$. We claim that $\Lambda_v$
satisfies the condition \eqref{eqn:conditions} of Theorem \ref{thm:conditions} for all $v$.
Indeed, let $K_v/k_v$ be a sub-$G$-extension whose inertia group $I_v$ injects into $\ZZ/Q\ZZ$. If $I_v = 0$ then $G_v$ is clearly cyclic. Otherwise $G_v$ is an extension of a cyclic group by $\ZZ/Q\ZZ$. Our assumptions on $G$ now imply that $G_v$ is also cyclic in this case,
which proves the claim. Applying Theorem \ref{thm:conditions} and using Proposition \ref{prop:solvable}, we
obtain the result.

Now suppose that the $Q$-Sylow subgroup of $G$ is not cyclic, 
so that $G$ contains a subgroup $A$ which is isomorphic to $(\ZZ/Q\ZZ)^2$.
Let $(K/k, \psi)$ be a $G$-extension such that $D_v$ is cyclic for all $v \notin S$. Clearly 
$\psi(D_v) \neq A$ for all $v \notin S$. However, such extensions have density $0$
in all $G$-extensions by Theorem~\ref{thm:bicyclic}, which proves the result in this case.
\qed

\section{The Hasse principle and weak approximation}\label{sec:HPandWA}
We now use our analytic results to prove  the remaining results on the Hasse norm principle and weak approximation,
as stated in the introduction.

\subsection{Norm one tori} \label{sec:tori}
We begin by recalling some results on norm one tori.
Let $K/k$ be a finite separable extension of fields. Consider the associated norm one torus $\Res_{K/k}^1 \Gm$,
given as the kernel of the norm map
$$\Res_{K/k}^1 \Gm := \ker( \Res_{K/k} \Gm \to \Gm),$$
where $\Res_{K/k} \Gm$ denotes the Weil restriction of $\Gm$ from $K$ to $k$.
If $k$ is a number field, we denote by $\Sha(\Res_{K/k}^1 \Gm)$
the Tate-Shafarevich group of $\Res_{K/k}^1 \Gm$ over $k$ (see \eqref{def:Sha}). Our interest in this stems from the following
well-known relation to the Hasse norm principle 
$$\Sha(\Res_{K/k}^1 \Gm) \cong \left(k^* \cap N_{K/k}(\Adele^*_K)\right)/ N_{K/k}(K^*)$$
(see e.g. \cite[p.~307]{PR94}).
We now recall how to calculate this group.
\begin{theorem}[Tate] \label{thm:Tate}
	Let $G$ be a finite group and let $(K/k, \psi)$ be a $G$-extension of number fields.
	Then there is a canonical isomorphism
	$$\QZdual{\Sha(\Res_{K/k}^1 \Gm)} \cong \ker\left( \mathrm{H}^3(G,\ZZ) \to \prod_{v} \mathrm{H}^3(\psi(D_v),\ZZ) \right).$$
\end{theorem}
\begin{proof}
	See e.g.~\cite[Ex.~5.6]{San81}.
\end{proof}
Here $D_v$ denotes the decomposition group of $K/k$ at $v$.
Note that if $H$ is a finite cyclic group, then $\mathrm{H}^3(H,\ZZ) = \mathrm{H}^1(H,\ZZ) = 0$. In particular,
it is only the ramified primes which are relevant in Theorem \ref{thm:Tate}.

Given a torus $T$ over a number field $k$, we write
$$A(T) = \left(\prod_vT(k_v)\right) /\overline{T(k)},$$
where $\overline{T(k)}$ denotes the closure of $T(k)$ in $\prod_vT(k_v)$
with respect to the product topology.
This group measures the failure of weak approximation for $T$. The next result encapsulates
the fact that the Brauer-Manin obstruction is the only obstruction to the Hasse principle
and weak approximation for the norm one tori $\Res_{K/k}^1 \Gm$, which is a special case
of a general result of Voskresenski\v{\i} \cite[Thm.~6]{Vos70}.

\begin{theorem}[Voskresenski\v{\i}] \label{thm:Sansuc}
	Let $G$ be a finite group and let $(K/k, \psi)$ be a $G$-extension of number fields.
	Then there is an exact sequence
	$$0 \to A(\Res_{K/k}^1 \Gm) \to \QZdual{\mathrm{H}^3(G, \ZZ)} \to \Sha(\Res_{K/k}^1 \Gm) \to 0.$$
	In particular, if $\mathrm{H}^3(G, \ZZ) \neq 0$ and $\Sha(\Res_{K/k}^1 \Gm) = 0$, then $\Res_{K/k}^1 \Gm$
	fails weak approximation.
\end{theorem}
\begin{proof}
	This follows from \cite[Thm.~6]{Vos70} and \cite[Prop.~7]{CTS77} (see also
	\cite[Ex.~5.6]{San81}).
\end{proof}



In our density results, we order norm one tori by the discriminant of the associated extension and count extensions of bounded discriminant. The following observation makes clear that we are only counting each such torus once this way.

\begin{proposition} \label{prop:tori}
Let $k$ be a number field and let $K_1/k$ and $K_2/k$ be Galois extensions. Then $\Res^1_{K_1/k}\Gm \cong \Res^1_{K_2/k}\Gm$ if and only if $K_1/k$ and $K_2/k$ are isomorphic extensions.
\end{proposition}

\begin{proof}
A Galois extension $K/k$ is determined up to $k$-isomorphism by the set of primes of $k$ 
which split completely in $K$ \cite[Cor.~13.19]{Neu99}. 
The proposition then follows from the fact that a prime $\fp$ of $k$ splits completely in $K$ if and only if the base-change of $\Res^1_{K/k}\Gm$ to $k_{\fp}$ is a split torus.
\end{proof}

\subsection{Proof of Theorem \ref{thm:solvable}}
Let $G$ be a finite solvable group. If $\mathrm{H}^3(G, \ZZ) = 0$, then $G$-extensions satisfy the Hasse norm principle
by Theorem \ref{thm:Tate}. On the other hand, if  $\mathrm{H}^3(G, \ZZ) \neq 0$ then let $K/k$ be as in
Proposition \ref{prop:solvable}. Then $\mathrm{H}^3(D_v, \ZZ) = 0$ for all $v$, hence the Hasse norm principle
fails by Theorem \ref{thm:Tate}. \qed

\subsection{Computations in group cohomology}

We now proceed with some cohomological computations. Lemma \ref{lem:exterior} and Lemma \ref{lem:exterior_calc}
are certainly well-known; we include proofs for completeness.

\begin{lemma} \label{lem:exterior}
	Let $G$ be a finite abelian group. Then there exists a canonical isomorphism
	$$\mathrm{H}^3(G,\ZZ) \cong \Hom(\exterior G,\QQ/\ZZ).$$
\end{lemma}
\begin{proof}
	The universal coefficient theorem yields an isomorphism
	$$\mathrm{H}^3(G,\ZZ) \cong \Ext(\mathrm{H}_2(G,\ZZ),\ZZ).$$
	Next, one has a canonical identification $\mathrm{H}_2(G,\ZZ) \cong \exterior G$ (see \cite[Thm.~V.6.4]{Bro82}), for example).
	Applying $\Hom(\exterior G, \cdot)$ to the exact sequence
	$$0 \to \ZZ \to \QQ \to \QQ/\ZZ \to 0,$$
	one obtains the exact sequence
	$$\Hom(\exterior G, \QQ) \to \Hom(\exterior G, \QQ/\ZZ) \to 
	\Ext(\exterior G,\ZZ) \to \Ext(\exterior G,\QQ).$$
	However, the first group is trivial as $\exterior G$ is finite,
	and the last group is trivial as $\QQ$ is divisible.
	This proves the result.
\end{proof}

The above result is very important for our work, as, together with Theorem \ref{thm:Tate},
it allows one to translate problems involving the Hasse norm principle
for abelian extensions into simple problems in linear algebra. We will use this duality between $\mathrm{H}^3(G,\ZZ)$ and $\exterior G$ 
implicitly throughout the rest of this paper.

\begin{lemma} \label{lem:exterior_calc}
	Let $n_{j+1} \mid n_j$ and  $$G = \ZZ/n_1\ZZ \oplus \cdots \oplus \ZZ/n_l\ZZ.$$
	Then
	$$\exterior{G} \cong (\ZZ/n_2\ZZ) \oplus \cdots \oplus  (\ZZ/n_j\ZZ)^{j-1} \oplus
	\cdots \oplus (\ZZ/n_{l}\ZZ)^{l-1}.$$
	In particular $\exterior G = 0$ if and only if $G$ is cyclic.
\end{lemma}
\begin{proof}
	We use the following facts:
	\begin{enumerate}
		\item $\exterior (G_1 \oplus G_2) = \exterior G_1 \oplus (G_1 \otimes G_2) \oplus \exterior G_2$.
		\item $\exterior (\ZZ/n\ZZ) = 0$.
		\item $\ZZ/n\ZZ \otimes \ZZ/m\ZZ \cong \ZZ/\gcd(n,m)\ZZ$.
	\end{enumerate}
	By induction, we therefore obtain
	$$\exterior G \cong \bigoplus_{i < j} (\ZZ/n_i\ZZ \otimes \ZZ/n_j\ZZ) \cong \bigoplus_{i < j} \ZZ/n_j\ZZ,$$
	which yields the result.
\end{proof}

The next lemma is the reason for the dichotomy which occurs between Theorem \ref{thm:HNP_0}
and Theorem \ref{thm:fail_HNP}.

\begin{lemma} \label{lem:special_form}
	Let $G$ be a finite abelian group and let $p$ be a prime dividing $|G|$.
	Then the exponent of $\exterior G$ divides $p$ if and only if $G\cong
	\ZZ/n\ZZ \oplus (\ZZ/p\ZZ)^r$ for some $n$ with $p\mid n$ and some $r \geq 0$.
\end{lemma}
\begin{proof}
	This follows easily from Lemma \ref{lem:exterior_calc}.
\end{proof}

Whilst the question of whether an abelian extension satisfies the Hasse norm principle is quite
complicated in general, it turns out that for weak approximation there is a simple local criterion.

\begin{lemma} \label{lem:WA}
	Let $K/k$ be an abelian extension of number fields. Then $\Res_{K/k}^1 \Gm$
	satisfies weak approximation if and only if the decomposition group of $K/k$ at each place of $k$ is cyclic.
\end{lemma}
\begin{proof}
	Let $G=\Gal(K/k)$.
	By Theorem \ref{thm:Tate} and Theorem \ref{thm:Sansuc}, it suffices to show that the equality
	$$ \mathrm{H}^3(G,\ZZ) = \ker\left(\mathrm{H}^3(G,\ZZ) \to \prod_v \mathrm{H}^3(D_v,\ZZ)\right)$$
	holds if and only if all $D_v$ are cyclic. For each place $v$, the homomorphism
	$\mathrm{H}^3(G,\ZZ) \to \mathrm{H}^3(D_v,\ZZ)$ is surjective, as the dual homomorphism $\exterior D_v \to \exterior G$
	is injective. The result then follows from the final assertion of Lemma \ref{lem:exterior_calc}, namely that  
	$\mathrm{H}^3(D_v,\ZZ) = 0$ if and only if $D_v$ is cyclic.
\end{proof}

\subsection{Proof of Theorem \ref{thm:abelian}}
This follows from Theorem \ref{thm:solvable} and Lemma \ref{lem:exterior_calc}, which show
that $\mathrm{H}^3(G, \ZZ) = 0$ if and only if $G$ is cyclic.
\qed

\subsection{Proof of Theorem  \ref{thm:WA}}
This follows immediately from Theorem \ref{thm:cyclic} and Lemma \ref{lem:WA}. \qed

\subsection{Proof of Theorem \ref{thm:satisfy_HNP}}
To prove Theorem \ref{thm:satisfy_HNP} it suffices to show the following.

\begin{proposition} \label{prop:satisfy_HNP}
	Let $k$ be a number field and let $G$ be a finite abelian group.
	Then there exists a $G$-extension $(K/k,\psi)$ such that $\Sha(\Res_{K/k}^1 \Gm) = 0$.
\end{proposition}

Indeed, assume that $G$ is non-cyclic and let $(K/k,\psi)$ be as in Proposition \ref{prop:satisfy_HNP}.
Then our assumptions, together with Theorem \ref{thm:Sansuc} and Lemma \ref{lem:exterior_calc}, imply that
\begin{equation} \label{eqn:fail_WA}
	A(\Res_{K/k}^1 \Gm) \neq 0, \quad \Sha(\Res_{K/k}^1 \Gm) = 0.
\end{equation}
Hence $K/k$ satisfies the Hasse norm principle and $\Res_{K/k}^1 \Gm$ fails weak approximation. 
 Moreover, 
let $S$ be the set of places where $K/k$ is ramified. Then applying Theorem \ref{thm:conditions} 
with $\Lambda_v = \{(K_v/k_v,\psi_v)\}$ for $v \in S$ and $\Lambda_v$ taken to be the set of all 
sub-$G$-extensions of $k_v$ for $v \not \in S$, by Theorem \ref{thm:Tate} we obtain the existence
of a  positive proportion of $G$-extensions of $k$ which also satisfy \eqref{eqn:fail_WA}. This is clearly
sufficient for Theorem \ref{thm:satisfy_HNP}.

We prove Proposition \ref{prop:satisfy_HNP} by finding sufficient local conditions for the Hasse
norm principle to hold. We begin with the following lemma (cf.~a result of Bogomolov ~\cite[Thm.~7.1]{CTS07}). We call an abelian group bicyclic if it is a product of two cyclic groups and is not cyclic.

\begin{lemma} \label{lem:G_i}
	Let $G$ be a finite abelian non-cyclic group. Then there exists a finite collection
	of bicyclic subgroups $G_i \subseteq G$ for $i \in I$, such that the natural map
	$$\bigoplus_{i \in I} \exterior G_i \to \exterior G$$
	is an isomorphism. Furthermore, if the exponent of $\exterior G$ divides a prime $p$, then all $G_i$ may be chosen
	isomorphic to $(\ZZ/p\ZZ)^2$.
\end{lemma}
\begin{proof}
	Write
	$$G = \ZZ/n_1\ZZ \oplus \cdots \oplus \ZZ/n_l\ZZ, \quad n_{j+1} \mid n_j$$
	and let 
	$$G_{i,j} = \ZZ/n_i\ZZ \oplus \ZZ/n_j\ZZ, \quad 1 \leq i < j \leq l.$$
	Then the natural map
	$$
		\varphi: \bigoplus_{1 \leq i < j \leq l} \exterior G_{i,j} \to \exterior G
	$$
	is surjective. Moreover, Lemma \ref{lem:exterior_calc} shows that the domain
	and codomain have the same cardinality, thus $\varphi$ is an isomorphism. 
	This proves the first part of the lemma. 
	If the exponent of $\exterior G$ divides a prime $p$, then a similar argument,
	combined with Lemma \ref{lem:special_form}, gives	the result.
\end{proof}

We now show that one can realise any bicyclic extension over \emph{some} local field.

\begin{lemma} \label{lem:bicyclic}
	Let $k$ be a number field and let $n,m \in \NN$. Let $v$ be a place
	of $k$ which is completely split in $k(\mu_{n})$. Then there exists an abelian
	extension $K_v/k_v$ whose Galois group is isomorphic to
	$$\ZZ/n\ZZ \times \ZZ/m\ZZ.$$
\end{lemma}
\begin{proof}
	Adjoin an $n$th root of a uniformiser to the unique unramified extension of $k_v$ of degree $m$.
\end{proof}

We now need to construct an abelian extension which realises certain given local extensions. 
As explained in the introduction, this need not always be possible.
However, the only problems which arise concern the prime $2$.
We shall content ourselves with the following well-known result, which may be deduced, for example,
from the methods and results of \cite[\S IX.2]{NSW00}.



\begin{proposition} \label{prop:GW}
	Let $S$ be a finite set of non-archimedean places of a number field $k$ and let $G$ be a finite abelian group.
 	For each $v	\in S$, let $(K_v / k_v, \psi_v)$ be a sub-$G$-extension. 
 	Assume that $S$ does not contain any primes which lie above $2$.
 	Then there exists a $G$-extension $(K/k,\psi)$ which has the given completions
 	$(K_v / k_v, \psi_v)$ for $v \in S$.	
\end{proposition}
Using this, we may now prove Proposition \ref{prop:satisfy_HNP}.

\begin{proof}[Proof of Proposition \ref{prop:satisfy_HNP}]
	Let $G_i$ be as in Lemma \ref{lem:G_i}. By Lemma \ref{lem:bicyclic},
	there exist distinct places $v_i \nmid 2$ of $k$
	and $G_i$-extensions $(K_{v_i}/k_{v_i}, \psi_{v_i})$.
	By Proposition \ref{prop:GW}, there exists a $G$-extension $(K/k,\psi)$ which realises
	these local extensions. Then
	by Lemma \ref{lem:exterior} and Lemma \ref{lem:G_i} the map
	$$\mathrm{H}^3(G,\ZZ) \to \prod_{v} \mathrm{H}^3(\psi_v(D_v),\ZZ),$$
	is injective. Hence $\Sha(\Res_{K/k}^1 \Gm) = 0$ by Theorem \ref{thm:Tate},	as required.
\end{proof}

As explained at the beginning of the section, this completes the proof of Theorem \ref{thm:satisfy_HNP}. \qed

\subsection{Proof of Theorem \ref{thm:HNP_0}}
We begin with the following lemma, which gives necessary local conditions for the Hasse norm principle
to fail.

\begin{lemma} \label{lem:HNP_Q}
	Let $k$ be a number field, let $n > 1$ and $r \geq 0$. Let $Q$ be the smallest prime
	dividing $n$ and let $G = \ZZ/n\ZZ \oplus (\ZZ/Q\ZZ)^r$. Then there exists a collection of subgroups $G_i \subseteq G$ for $i \in I$ that are each isomorphic to $(\ZZ/Q\ZZ)^2$,
	with the following property. If a $G$-extension $(K/k,\psi)$ fails the Hasse norm principle, 
	then there exists some $i \in I$ such that $\psi(D_v) \neq G_i$ for all places $v$ of $k$.
\end{lemma}
\begin{proof}
	By Lemma \ref{lem:special_form}, the exponent of $\exterior{G}$ divides $Q$.
	We then choose the $G_i$ as in the second part of Lemma \ref{lem:G_i}. Let $(K/k, \psi)$
	be a $G$-extension which fails the Hasse norm principle. If for each
	$i \in I$ there exists some $v$ such that $\psi(D_v) = G_i$, then by Theorem \ref{thm:Tate}
	we see that $\Sha(\Res_{K/k}^1 \Gm) = 0$; a contradiction.
\end{proof}
Using Lemma \ref{lem:HNP_Q}, we therefore obtain
\begin{align*}
	&\#\{(K/k,\psi) \in \gextk: \Delta(K/k) \leq B,  \, \Sha(\Res_{K/k}^1 \Gm) \neq 0 \}\\
	&\leq \#\{(K/k,\psi) \in \gextk : \Delta(K/k) \leq B, \, \exists i \in I \text{ s.t. } \psi(D_v) \neq G_i \forall v \} \\
	&\leq \sum_{i \in I} \#\{(K/k,\psi) \in \gextk : \Delta(K/k) \leq B, \, \psi(D_v) \neq G_i \forall v \}.
\end{align*}
As $G_i \cong (\ZZ/Q\ZZ)^2$, we may now apply Theorem \ref{thm:bicyclic} to obtain Theorem \ref{thm:HNP_0}. \qed

\subsection{Proof of Theorem \ref{thm:fail_HNP}}
The following lemma gives sufficient local conditions for the Hasse norm principle
to fail.
\begin{lemma} \label{lem:fail_HNP}
	Let $G$ be as in Theorem \ref{thm:fail_HNP} and let $(K/k, \psi)$ be a $G$-extension.
	Suppose that for all places $v$, either the inertia
	group at $v$ has order dividing $Q$ or the decomposition group $D_v$
	is cyclic. Then $K/k$ fails the Hasse norm principle.
\end{lemma}
\begin{proof}
	By Lemma \ref{lem:special_form}, we see that $\exterior{G}$ has a non-zero element whose order is not $Q$.
	Let $v$ be a place of $k$. If $D_v$ is cyclic, then $\exterior D_v=0$. Otherwise, by our assumptions, we have
	$D_v \cong \ZZ/Q\ZZ \oplus H$ for some cyclic group $H$ with $Q \mid |H|$. Thus by Lemma \ref{lem:exterior_calc},
	we obtain $\exterior{D_v} \cong \ZZ/Q\ZZ$. Hence the map
	$$\prod_v \exterior{(\psi(D_v))} \to \exterior{G}$$
	is not surjective. Therefore its dual is not injective, and so the Hasse norm principle fails by Theorem \ref{thm:Tate}.
\end{proof}

Theorem \ref{thm:fail_HNP} is now an immediate consequence of Theorem~\ref{thm:conditions}, Proposition~\ref{prop:solvable}, and Lemma \ref{lem:fail_HNP}. \qed

\end{document}